\documentclass[12pt]{amsart}
\usepackage[centertags]{amsmath}
\usepackage{amsfonts}
\usepackage{amssymb}
\usepackage{amsthm}
\usepackage{latexsym}
\usepackage{amsmath}
\usepackage{amsfonts}

\numberwithin{equation}{section}

\newtheorem{theorem}{Theorem}[section]

\newtheorem{lemma}[theorem]{Lemma}
\newtheorem{corollary}[theorem]{Corollary}

\theoremstyle{definition}

\theoremstyle{remark}
\newtheorem{remark}[theorem]{Remark}

\newcommand{\rarrow}{\rightarrow}
\newcommand{\dis}{\displaystyle}
\newcommand{\be}{\begin{equation}}
\newcommand{\ee}{\end{equation}}
\newcommand{\bea}{\begin{eqnarray*}}
\newcommand{\eea}{\end{eqnarray*}}
\newcommand{\pa}{\partial}

\title{Inverse problem for a structural acoustic interaction}
\author[Shitao Liu]{Shitao Liu \\ [3mm] Department of Mathematics \\ University of Virginia \\ Charlottesville, VA 22904, USA 
\\ Email: sl3fa@virginia.edu}
\date{}

\begin{document}

\begin{abstract}
In this work, we consider an inverse problem of determining a source term for a structural acoustic partial differentia equation (PDE) model,
comprised of a two or three-dimensional interior acoustic wave equation coupled to a Kirchoff plate equation, with the coupling being accomplished
across a boundary interface. For this PDE system, we obtain the uniqueness and stability estimate for the source term from a single measurement
of boundary values of the ``structure''. The proof of uniqueness is based on Carleman estimate. Then, by means of an observability inequality and a
compactness/uniqueness argument, we can get the stability result. Finally, an operator theoretic approach gives us the regularity needed for the
initial conditions in order to get the desired stability estimate.
\end{abstract}

\maketitle

\section{Introduction and Main Results}

\subsection{Statement of the Problem}

\medskip
Let $\Omega$ be an open bounded subset of $\mathbb{R}^2$ or $\mathbb{R}^3$ with
smooth boundary $\Gamma$ of class $C^2$, and we designate a nonempty
simply connected segment of $\Gamma$ as $\Gamma_0$ with then $\Gamma=\Gamma_0\cup\Gamma_1$ and $\Gamma_0\cap\Gamma_1=\emptyset$.
We consider here the following system comprised of a ``coupling'' between a wave equation and an elastic plate-like equation:
\begin{equation}\label{nonlinear}
\begin{cases}
z_{tt}(x,t) = \Delta z(x,t) + q(x)z(x,t) & \mbox{in } \Omega \times [0,T] \\
\frac{\pa z}{\pa\nu}(x,t) = 0 & \mbox{on } \Gamma_1\times[0,T] \\
z_t(x,t) = -v_{tt}(x,t)-\Delta^2 v(x,t)-\Delta^2 v_t(x,t) & \mbox{on } \Gamma_0\times[0,T] \\
v(x,t)=\frac{\pa v}{\pa \nu}(x,t)=0 & \mbox{on } \pa\Gamma_0\times[0,T] \\
\frac{\pa z}{\pa \nu}(x,t)=v_t(x,t) & \mbox{on } \Gamma_0\times[0,T] \\
z(\cdot,\frac{T}{2}) = z_0(x) & \mbox{in } \Omega \\
z_t(\cdot,\frac{T}{2}) = z_1(x) & \mbox{in } \Omega \\
v(\cdot,\frac{T}{2}) = v_0(x) & \mbox{on } \Gamma_0 \\
v_t(\cdot,\frac{T}{2}) = v_1(x) & \mbox{on } \Gamma_0
\end{cases}
\end{equation}
where the coupling occurs across the boundary interface $\Gamma_0$.
$[z_0,z_1,v_0,v_1]$ are the given initial conditions and $q(x)$ is a time-independent unknown coefficient.
For this system, notice that the map $\{q\}\to \{z(q),v(q)\}$ is nonlinear, therefore we consider the following
\textit{nonlinear inverse problem}:
Let $\{z=z(q), v=v(q)\}$ be the weak solution to system $\eqref{nonlinear}$. Under suitable geometrical
conditions on $\Gamma_1=\Gamma\setminus\Gamma_0$, is it possible to retrieve $q(x)$,
$x\in\Omega$, from measurement of $v_{tt}(q)$ on $\Gamma_0\times[0,T]$?
In other words, is it possible to recover the internal wave potential from
the observation of the acceleration of the elastic plate.

Our emphasis here that we determine the interior acoustic property from observing the acceleration of the
elastic wall (portion of the boundary), is not only due to physical consideration, but also to the implications
of such inverse type analysis related to the coupling nature of the structural acoustic flow. In many structural
acoustics applications, the problem of controlling interior acoustic properties is directly correlated with the
problem of controlling structural vibrations since the interior noise fields are often generated by the vibrations
of an enclosing structure. An important example of this is the problem of controlling interior aircraft
cabin noise which is caused by fuselage vibrations that are induced by the low frequency high magnitude exterior
noise fields generated by the engines.

The primary goal in this paper is to study the uniqueness and stability of
the interior time-independent unknown coefficient $q(x)$ in some appropriate function space.
More precisely, we consider the follow uniqueness and stability problems:

\medskip
\textbf{Uniqueness in the nonlinear inverse problem}

Let $\{z=z(q), v=v(q)\}$ be the weak solution to system $\eqref{nonlinear}$. Under
geometrical conditions on $\Gamma_1$, does the acceleration of the wall $v_{tt}|_{\Gamma_0\times
[0,T]}$ determine $q(x)$ uniquely? In other words, does
$$v_{tt}(q)|_{\Gamma_0\times [0,T]}=v_{tt}(p)|_{\Gamma_0\times [0,T]}$$
imply $q(x)=p(x)$ in $\Omega$?

\medskip
\textbf{Stability in the nonlinear inverse problem}

Let $\{z(q), v(q)\}$, $\{z(p), v(p)\}$ be weak solutions to system $\eqref{nonlinear}$ with corresponding coefficients $q(x)$ and $p(x)$.
Under geometric conditions on $\Gamma_1$, is it possible to estimate
$\dis\|q-p\|_{L^2(\Omega)}$ by some suitable norms of $\dis (v_{tt}(q)-v_{tt}(p))|_{\Gamma_0\times[0,T]}$?

In order to study the \textit{nonlinear inverse problem}, we first linearize $\eqref{nonlinear}$
and hence we consider the following system:
\begin{equation}\label{linear}
\begin{cases}
w_{tt}(x,t) - \Delta w(x,t) - q(x)w(x,t) = f(x)R(x,t) & \mbox{in } \Omega \times [0,T] \\
\frac{\pa w}{\pa\nu}(x,t) = 0 & \mbox{on } \Gamma_1\times[0,T] \\
w_t(x,t) = -u_{tt}(x,t)-\Delta^2 u(x,t)-\Delta^2 u_t(x,t) & \mbox{on } \Gamma_0\times[0,T] \\
u(x,t)=\frac{\pa u}{\pa \nu}(x,t)=0 & \mbox{on } \pa\Gamma_0\times[0,T] \\
\frac{\pa w}{\pa \nu}(x,t)=u_t(x,t) & \mbox{on } \Gamma_0\times[0,T] \\
w(\cdot,\frac{T}{2}) = 0 & \mbox{in } \Omega \\
w_t(\cdot,\frac{T}{2}) = 0 & \mbox{in } \Omega \\
u(\cdot,\frac{T}{2}) = 0 & \mbox{on } \Gamma_0 \\
u_t(\cdot,\frac{T}{2}) = 0 & \mbox{on } \Gamma_0
\end{cases}
\end{equation}
where $q\in L^{\infty}(\Omega)$ is given, $R(x,t)$ is fixed suitably
while $f(x)$ is an unknown time-independent coefficient. For this linearized system,
we have the advantage that the map $\{f\}\to \{w(f),u(f)\}$ is linear, hence
we consider the corresponding \textit{linear inverse problem}:

\medskip
\textbf{Uniqueness in the linear inverse problem}

Let $\{w=w(f), u=u(f)\}$ be the weak solution to system $\eqref{linear}$. Under geometrical
conditions on $\Gamma_1$, does $u_{tt}|_{\Gamma_0\times [0,T]}$
determine $f(x)$ uniquely? In other words, does
$$u_{tt}(f)|_{\Gamma_0\times [0,T]}=0$$ imply $f(x)=0$ in $\Omega$?

\medskip
\textbf{Stability in the linear inverse problem}

Let $\{w=w(f), u=u(f)\}$ be the weak solution to system $\eqref{linear}$. Under
geometrical conditions on $\Gamma_1$, is it possible
to estimate $\dis\|f\|_{L^2(\Omega)}$ by some suitable norms of $\dis u_{tt}|_{\Gamma_0\times[0,T]}$?
\begin{remark}
In our models $\eqref{nonlinear}$ and $\eqref{linear}$ we regard
$t=\frac{T}{2}$ as the initial time. This is not essential, because
the change of independent variables $t\to t-\frac{T}{2}$ transforms
$t=\frac{T}{2}$ to $t=0$. However, this is convenient for us to
apply the Carleman estimate established in \cite{l-t-z}. In fact, one can
keep $t=0$ as initial moment by doing an even extension of $w$ and
$u$ to $\Omega\times[-T, T]$, but then the Carleman estimate in \cite{l-t-z}
needs to be modified accordingly.
\end{remark}

\subsection{Literature and Motivation}

The PDE system $\eqref{nonlinear}$ is an example of a \emph{structural acoustic interaction}. It mathematically
describes the interaction of a vibrating beam/plate in an enclosed acoustic field or chamber. In this situation, the
boundary segment $\Gamma_1$ represents the ``hard'' walls of the chamber $\Omega$, with $\Gamma_0$ being the flexible
portion of the chamber wall. The flow with in the chamber is assumed to be of acoustic wave type, and hence the presence
of the wave equation in $\Omega$, satisfied by $z$ in $\eqref{nonlinear}$, coupled to a structural plate equation
(in variable $v$) on the flexible boundary portion $\Gamma_0$. This type of PDE models has long existed in the literature
and has been an object of intensive experimental and numerical studies at the Nasa Langley Research Center \cite{l-l,
b-f-s-s,b-s}. Moreover, recent innovations in smart material technology and the potential applications of these innovations
in control engineering design have greatly increased the interest in studying these structural acoustic models. As a result,
there has been a lot of recent contributions to the literature deal with various topis; e.g., optimal control, stability,
controllability, regularity \cite{a.1,a.2,a-l.1,a-l.2,a-l.3,a-l.4,a-l-r.1,a-l-r.2,c-t,l}. However, to the best
of our knowledge, there are no results available in the literature for our inverse type analysis on the model.

On the other hand, the interest to the inverse problem has been stimulated by the studies of applied problems such as geophysics,
medical imaging, scattering, nondestructive testing and so on. These problems are of the determination of unknown coefficients
of differential equations which are the functions depending on the point of the space \cite{b,is.1,is.2}. For the uniqueness in multidimensional inverse problem
with a single boundary observation, the pioneering paper by Bukhgeim and Klibanov \cite{b-k} provides a methodology based on a type of exponential weighted energy estimate,
which is usually referred as the Carleman estimate since the original work \cite{c} by Carleman.
After \cite{b-k}, several papers concerning inverse problems by using Carleman estimate have been published (e.g. \cite{is.3,k.1}).
In particular, for the inverse hyperbolic type problems that is related to our concern in this paper, there has been intensively studies \cite{i-y,is-y,kh,p-y,y}. However,
we mentioned again that there is not any such uniqueness and stability analysis for the structural acoustic models or even in general coupled PDE systems.
This motivates the work of the present paper.

The usual problem setting for inverse hyperbolic problem includes determining a coefficient from measurements on the whole boundary or part of the boundary, either Dirichlet type
\cite{b-k,kh,p-y,y} or Neumann type \cite{i-y,is-y}. Usually the coefficient describes a physical property of the medium (e.g. the elastic modulus in Hooke's law), and the inverse problem is
to determine such a property. In our formulation of the inverse problem, we need to determine the time-independent wave potential $q(x)$ by observing the acceleration from the
flexible portion of the boundary $\Gamma_0$. The mathematical challenge in this problem stems from the fact that we are dealing with the ``coupling'' on the part of the boundary
and the main technical difficulty associated with this structure is the lack of the compactness of the resolvent. As a result, the space regularity for the solution of the
wave equation component is limited by the structure on the plate and hence this will prevent us going to higher dimension ($n>7$)
no matter how smooth the initial data is. This is a distinguished feature of this structural acoustic model comparing to the
purely wave equation model as in that case the solution can be as smooth as we want as long as the initial data is smooth enough.
In this present paper, we prove the cases where the dimension $n=2$ and $3$ (physical meaningful cases) by using the Carleman estimate for the Neumann problem in \cite{l-t-z} and an operator theoretic formulation.
We show that indeed the observation of the acceleration on the plate can determine the potential $q$ under some
restrictions on the initial data and some geometrical conditions on the boundary. As we mentioned, the argument will also work
for dimension up to $n=7$.

\subsection{Main Assumptions and Preliminaries}

In this section we state the main geometrical assumptions throughout this paper. These assumptions are essential in order to establish the
Carleman estimate stated in section 2.

Let $\nu=[\nu_1,\cdots,\nu_n]$ be the unit outward normal vector on $\Gamma$, and let $\frac{\partial}{\partial \nu}=\nabla\cdot\nu$ denote the
corresponding normal derivative
Moreover, we assume the following geometric conditions
on $\Gamma_1=\Gamma\setminus\Gamma_0$:

\medskip
(A.1) There exists a strictly convex (real-valued) non-negative
function $\dis d:\overline{\Omega}\to\mathbb{R}^+$, of class
$C^3(\overline{\Omega})$, such that, if we introduce the
(conservative) vector field $h(x)=[h_1(x),\cdots,h_n(x)]\equiv\nabla
d(x), x\in\Omega$, then the following two properties hold true:

\medskip
(i)\be\label{assume1} \frac{\pa d}{\pa\nu}\bigg|_{\Gamma_1}=\nabla
d\cdot\nu=h\cdot\nu=0; \quad h\equiv\nabla d \ee
(ii) the (symmetric) Hessian matrix $\mathcal{H}_d$ of $d(x)$ [i.e.,
the Jacobian matrix $J_h$ of $h(x)$] is strictly positive definite
on $\overline{\Omega}$: there exists a constant $\rho>0$ such that
for all $x\in\overline{\Omega}$: \be\label{assume2}
\mathcal{H}_d(x)=J_h(x)=\left[\begin{array}{ccc}
d_{x_1x_1} & \cdots & d_{x_1x_n} \\
\vdots &  & \vdots \\
d_{x_nx_1} & \cdots & d_{x_nx_n} \\
\end{array}\right]
=\left[\begin{array}{ccc}
\frac{\pa h_1}{\pa x_1} & \cdots & \frac{\pa h_1}{\pa x_n} \\
\vdots &  & \vdots \\
\frac{\pa h_n}{x_1} & \cdots & \frac{\pa h_n}{\pa x_n} \\
\end{array}\right]
\geq\rho I \ee
(A.2) $d(x)$ has no critical point on $\overline{\Omega}$:
\be\label{assume3} \inf_{x\in\Omega}|h(x)|=\inf_{x\in\Omega}|\nabla
d(x)|=s>0 \ee
\begin{remark}
One canonical example is that $\Gamma_1$ is flat (not the case in our problem setting here), where then we can take
$d(x)=|x-x_0|^2$, with $x_0$ on the hyperplane containing $\Gamma_1$ and outside $\Omega$, then $h(x)=\nabla d(x)=
2(x-x_0)$ is radial. However, in general $h(x)$ is not necessary radial. In particularly in our case where $\Gamma_1$
is convex, the corresponding required $d(x)$ can also be explicitly constructed. For more examples of such function
$d(x)$ with different geometries of $\Gamma_1$, we refer to the appendix of \cite{l-t-z}.
\end{remark}
Next we introduce an abstract operator theoretic formulation associated to $\eqref{nonlinear}$ for which we will need the following facts and definitions:
Let the operator $A$ be
\be\label{operatorA}
Az=-\Delta z-q(x)z,\quad D(A)=\{z: \Delta z+q(x)z\in L^2(\Omega), \frac{\pa
z}{\pa \nu}\bigg|_{\Gamma}=0\}\ee
Notice the lower-order part is a perturbation which preserves generation of the self-adjoint principle part $A_N$ (e.g. \cite{l-t.4}),
where $A_N:L^2(\Omega)\supset D(A_N)\rarrow L^2(\Omega)$ is
defined by:
\be\label{operatorAN}
A_Nz=-\Delta z,\quad D(A_N)=\{z: \Delta z\in L^2(\Omega), \frac{\pa
z}{\pa \nu}\bigg|_{\Gamma}=0\}\ee
Then $A_N$ is positive self-adjoint and
\be\label{AN1/2}
D(A_N^{\frac{1}{2}})=H^1_{\Gamma_1}(\Omega)=\{z: z\in
H^1(\Omega), \frac{\pa z}{\pa\nu}=0 \ \text{on} \ \Gamma_1\}\ee
Then we define the Neumann map $N$ by:
\be z=Ng\;\Leftrightarrow\;\begin{cases}
\Delta z=0&\text{in}\;\;\;\Omega\\
\frac{\pa z}{\pa\nu} =0 &\text{on}\;\;\;\Gamma_1\\
\frac{\pa z}{\pa \nu}=g&\text{on}\;\;\;\Gamma_0
\end{cases}\ee
By elliptic theory
\begin{equation}
N\in\mathcal{L}(L^2(\Gamma_0),H^{3/2}_{\Gamma_1}(\Omega))
\end{equation}
Now we define
\begin{equation}
\mathcal{B}=A_NN:L^2(\Gamma_0)\rightarrow
D(A_N^{1\over2})'
\end{equation}
via the conjugation $\mathcal{B}^{*}=N^{*}A_N$. Then with $v\in L^2(\Gamma)$ and for any $y\in D(A_N^{1\over2})$ we have
\begin{multline}
-(\mathcal{B}^{*}y,v)_{\Gamma}=-(N^*A_Ny,v)_{\Gamma}=-(A_Ny,Nv)_{\Omega}=(\Delta y,Nv)_{\Omega}\\
=(y,\Delta(Nv))_{\Omega}+(\frac{\pa y}{\pa\nu},Nv)_{\Gamma}-(y,\frac{\pa(Nv)}{\pa\nu})_{\Gamma}=-(y,v)_{\Gamma_0}
\end{multline}
by Green's theorem, the definition of $N$ and the fact $\dis\frac{\pa y}{\pa\nu}=0$ on $\Gamma_1$ when $y\in D(A_N^{1\over2})$. In other words, we have
\be
N^{*}A_Ny=\begin{cases}
y,&\text{on}\;\;\;\Gamma_0\\
0,&\text{on}\;\;\;\Gamma_1
\end{cases}\;\;\;\;\;\text{for}\;y\in D(A_N^{1\over2})\ee
i.e. $\mathcal{B}^{*}=N^{*}A_N$ is the restriction of the trace
map from
$H^1(\Omega)$ to $H^{\frac{1}{2}}(\Gamma_0)$.

Last we set $\textbf{\AA}:L^2(\Gamma_0)\supset D(\textbf{\AA})\rightarrow
L^2(\Gamma_0)$ to be
\be\label{laplace2}
\textbf{\AA}=\Delta^2,D(\textbf{\AA})=\{v\in H^2_0(\Gamma_0):\Delta^2v\in L^2(\Gamma_0)\} \ee
where $H^2_0(\Gamma_0)=\{v\in H^2(\Omega):v=\frac{\pa v}{\pa\nu}=0 \ \text{on} \ \pa\Gamma_0\}$.
$\textbf{\AA}$ is self-adjoint, positive definite, and we have the characterization
\be\label{domainA} D(\textbf{\AA}^{\frac{1}{2}})=H^2_0(\Gamma_0) \ee Now set
\begin{equation}\mathcal{A}=\left[\begin{array}{cccc}
                      0 & I & 0 & 0 \\
                      -A_N+q & 0 & 0 &
\mathcal{B} \\
                      0 & 0 & 0 & I \\
    0 & -\mathcal{B}^* & -\textbf{\AA} & -\textbf{\AA} \end{array}
\right]\label{A}\end{equation}
on the energy space
\begin{equation}
\begin{split}
H & = D(A_N^{\frac{1}{2}})\times L^2(\Omega) \times D(\textbf{\AA}^{\frac{1}{2}}) \times L^2(\Gamma_0) \\
  & = H^1_{\Gamma_1}(\Omega)\times L^2(\Omega) \times H^2_0(\Gamma_0) \times L^2(\Gamma_0)
\end{split}
\end{equation}
Then we have the domain of the operator $\mathcal{A}$
\begin{equation}\label{DomA}
\begin{split}
D(\mathcal{A})& =\{[z_0,z_1,v_0,v_1]^T\in [D(A_N^{\frac{1}{2}})]^2 \times [D(\textbf{\AA}^{\frac{1}{2}})]^2 \ \text{such that} \\
& \qquad -z_0+Nv_1\in D(A_N) \ \text{and} \ v_0+v_1\in D(\textbf{\AA})\} \\
              & =\{[z_0,z_1,v_0,v_1]^T: z_0\in H^1_{\Gamma_1}(\Omega), z_1\in H^1_{\Gamma_1}(\Omega), v_0\in H^2_0(\Gamma_0), v_1\in H^2_0(\Gamma_0), \\
& \qquad (\Delta+q)z_0\in L^2(\Omega), \frac{\pa z_0}{\pa\nu}=v_1 \ \text{on} \ \Gamma_0 \ \text{and} \ v_0+v_1\in D(\textbf{\AA})\} \\
              & =\{[z_0,z_1,v_0,v_1]^T: z_0\in H^2_{\Gamma_1}(\Omega), z_1\in H^1_{\Gamma_1}(\Omega), v_0\in H^2_0(\Gamma_0), v_1\in H^2_0(\Gamma_0), \\
& \qquad \frac{\pa z_0}{\pa\nu}=v_1 \ \text{on} \ \Gamma_0 \ \text{and} \ v_0+v_1\in D(\textbf{\AA})\}
\end{split}
\end{equation}
where in the last step we get $z_0\in H^2(\Omega)$ from $q\in L^{\infty}(\Omega)$ and $(\Delta+q)z_0\in L^2(\Omega)$ due to elliptic theory.
Therefore with these notations, the original system $\eqref{nonlinear}$ becomes to the first order abstract differential equation
\be
\frac{dy}{dt}=\mathcal{A}y\ee
where $y=[z,z_t,v,v_t]^T$.
From semigroup theory, when the initial conditions $[z_0,z_1,v_0,v_1]$ are in $D(\mathcal{A})$ we have that the solution $y$ satisfies
\be
y\in D(\mathcal{A}), \quad y_t\in H
\ee
\begin{remark}
The structure of $\mathcal{A}$ reflects the coupled nature of this structural acoustic system $\eqref{nonlinear}$.
One distinguished feature of the system is that the resolvent of $\mathcal{A}$ is not compact. However, it can still
be shown that $\mathcal{A}$ generates a $C_0$-semigroup of contractions $\{e^{\mathcal{A}t}\}_{t\geq0}$ which
establishes the well-posedness of the system \cite{a-l.2}.
\end{remark}

\subsection{Main results}

For the inverse problems stated in section 1.1, we have the following results:
\begin{theorem}\label{th1}(Uniqueness for the linear inverse problem)
Under the main assumptions (A.1), (A.2)
and let \be\label{time} T>2\sqrt{\max_{x\in\overline{\Omega}}d(x)}
\ee Moreover, let \be\label{regR} R\in W^{3,\infty}(Q) \ee and
\be\label{crucialR}
\bigg|R\left(x,\frac{T}{2}\right)\bigg|\geq r_0>0,\qquad \bigg|R_t\left(x,\frac{T}{2}\right)\bigg|\geq r_1>0 \ee
for some positive constants
$r_0$, $r_1$ and $x\in\overline{\Omega}$. In addition, let
\be\label{regq} q\in
L^{\infty}(\Omega) \ee If the weak solution $\{w=w(f), u=u(f)\}$ to system
$\eqref{linear}$ satisfies
\be\label{h2reg}
w,w_t,w_{tt}\in H^{2}(Q)=H^2(0,T'L^2(\Omega))\cap L^2(0,T;H^2(\Omega))
\ee and
\be\label{linearuniqueness} u_{tt}(f)(x,t)=0,
\quad x\in\Gamma_0, t\in[0,T] \ee then $f(x)=0$,  $x\in\Omega$.
\end{theorem}
\begin{theorem}\label{th2}(Uniqueness for the nonlinear inverse problem)
Under the main assumptions (A.1), (A.2), assume
$\eqref{time}$ and \be\label{regqp} q,p\in L^{\infty}(\Omega) \ee
Let either of $z(q)$ and $z(p)$ satisfy \be\label{regw} z\in
W^{3,\infty}(Q) \ee Moreover, let \be\label{crucialz}|z_0(x)|\geq
s_0>0,\qquad |z_1(x)|\geq s_1>0 \ee for some positive constants
$s_0$, $s_1$ and $x\in\overline{\Omega}$. If the weak solutions $\{z(q), v(q)\}$ and $\{z(p), v(p)\}$ to system $\eqref{nonlinear}$ satisfy
\be\label{differenceh2}
z(q)-z(p), z_t(q)-z_t(p), z_{tt}(q)-z_{tt}(p)\in H^{2}(Q)
\ee
and
\be\label{nonlinearuniqueness} v_{tt}(q)(x,t)=v_{tt}(p)(x,t), \quad
x\in\Gamma_0, t\in[0,T] \ee then $q(x)=p(x)$, $x\in\Omega$.
\end{theorem}
\begin{theorem}\label{th3}(Stability for the linear inverse problem)
Under the main assumptions (A.1), (A.2), assume $\eqref{time}$, $\eqref{regR}$, $\eqref{crucialR}$
and $\eqref{regq}$. Moreover, let
\be\label{regRt}
R_{t}\in H^{\frac{1}{2}+\epsilon}(0,T;L^{\infty}(\Omega))\ee
for some $0<\epsilon<\frac{1}{2}$.
Then there exists a constant $C=C(\Omega,T,\Gamma_0,\varphi,q,R)>0$
such that
\begin{equation}\label{stability}
\|f\|_{L^2(\Omega)}\leq C\left(\|u_{tt}\|_{L^2(\Gamma_0\times[0,T])}+\|u_{ttt}\|_{L^2(\Gamma_0\times[0,T])}+\|\Delta^2u_{tt}\|_{L^2(\Gamma_0\times[0,T])}\right)
\end{equation}
for all $f\in L^2(\Omega)$.
\end{theorem}
\begin{theorem}\label{th4}(Stability for the nonlinear inverse problem)
Under the main assumptions (A.1), (A.2), assume
$\eqref{time}$, $\eqref{regqp}$, $\eqref{regw}$ and $\eqref{crucialz}$. Moreover,
let the initial data satisfy the compatibility condition
\begin{enumerate}
\item
When $n=2$, $[z_0,z_1,v_0,v_1]\in D(\mathcal{A}^2)$ where
\begin{equation*}
\begin{split}
D(\mathcal{A}^2) & =\{[z_0,z_1,v_0,v_1]^T: z_0\in H^3_{\Gamma_1}(\Omega),z_1\in H^2_{\Gamma_1}(\Omega),v_0\in H^2_0(\Gamma_0), v_1\in H^2_0(\Gamma_0), \\
& \qquad\textbf{\AA}(v_0+v_1)+B^{*}z_1\in H^2_0(\Gamma_0), v_1-\textbf{\AA}(v_0+v_1)-B^{*}z_1\in D(\textbf{\AA}), \\
& \qquad\frac{\pa z_0}{\pa\nu}|_{\Gamma_0}=v_1, \frac{\pa z_1}{\pa\nu}|_{\Gamma_0}=-\textbf{\AA}(v_0+v_1)-B^{*}z_1 \}
\end{split}
\end{equation*}
\item
When $n=3$, $[z_0,z_1,v_0,v_1]\in D(\mathcal{A}^3)$ where
\begin{equation*}
\begin{split}
D(\mathcal{A}^3) & =\{[z_0,z_1,v_0,v_1]^T: z_0\in H^{\frac{7}{2}}_{\Gamma_1}(\Omega), z_1\in H^3_{\Gamma_1}(\Omega), v_0\in H^2_0(\Gamma_0), v_1\in H^2_0(\Gamma_0), \\
& \textbf{\AA}(v_0+v_1)+B^{*}z_1\in H^2_0(\Gamma_0), \frac{\pa z_0}{\pa\nu}|_{\Gamma_0}=v_1, \frac{\pa z_1}{\pa\nu}|_{\Gamma_0}=-\textbf{\AA}(v_0+v_1)-B^{*}z_1 \\
& \textbf{\AA}(v_0+v_1)+B^{*}z_1+\textbf{\AA}[v_1-\textbf{\AA}(v_0+v_1)-B^{*}z_1]+B^{*}[(-A_N+q)z_0+Bv_1]\in D(\textbf{\AA}) \\
& \textbf{\AA}(v_1-\textbf{\AA}(v_0+v_1)-B^{*}z_1)+B^{*}[(-A_N+q)z_0+Bv_1]\in H^2_0(\Gamma_0), \\
& \frac{\pa [(-A_N+q)z_0+Bv_1]}{\pa\nu}|_{\Gamma_0}=-\textbf{\AA}[v_1-\textbf{\AA}(v_0+v_1)-B^{*}z_1]-B^{*}[(-A_N+q)z_0+Bv_1]\}
\end{split}
\end{equation*}
\noindent Then there exists a constant
$C=C(\Omega,T,\Gamma_0,\varphi,q,p,z_0,z_1,v_0,v_1)>0$ such that
\begin{multline}\label{nonlinearstability} \|q-p\|_{L^2(\Omega)}\leq
C\left(\|v_{tt}(q)-v_{tt}(p)\|_{L^2(\Gamma_0\times[0,T])}+\|v_{ttt}(q)-v_{ttt}(p)\|_{L^2(\Gamma_0\times[0,T])}\right. \\
\left. \qquad +\|\Delta^2(v_{tt}(q)-v_{tt}(p))\|_{L^2(\Gamma_0\times[0,T])}\right)
\end{multline} for all $q,p\in W^{1,\infty}(\Omega)$ when $n=2$ and all $q,p\in W^{2,\infty}(\Omega)$ when $n=3$.
\end{enumerate}
\end{theorem}

\medskip
The rest of this paper is organized as follows: In section 2 we give the key Carleman estimate that is used in the proof of uniqueness result.
Based on the same Carleman estimate, we also prove an observability inequality that is needed in section 5.
Section 3 to 6 are devoted to the proofs of our main results Theorems \ref{th1} to \ref{th4}. Some concluding remarks will
be given in section 7.

\section{Carleman estimate and observability inequality}

\subsection{Carleman Estimate}

In this section, we state a Carleman estimate result that plays a key
role in the proof of our uniqueness theorem. The result is due to \cite{l-t-z}.

We first introduce the pseudo-convex function $\varphi(x,t)$ defined by \be\label{defphi}
\varphi(x,t)=d(x)-c\left(t-\frac{T}{2}\right)^2; \quad x\in\Omega,
t\in[0,T]\ee where $T$ is as in $\eqref{time}$ and $0<c<1$ is
selected as follows:
By $\eqref{time}$, there exists $\delta>0$ such that
\be\label{timesquare} T^2>4\max_{x\in\overline{\Omega}}d(x)+4\delta
\ee For this $\delta>0$, there exists a constant $c$, $0<c<1$, such
that \be\label{ctsquare}
cT^2>4\max_{x\in\overline{\Omega}}d(x)+4\delta\ee Henceforth, with
$T$ and $c$ chosen as described above, this function $\varphi(x,t)$
has the following properties:

(a) For the constant $\delta>0$ fixed in $\eqref{timesquare}$ and for any $t>0$
\be\label{propertya}
\varphi(x,t)\leq\varphi(x,\frac{T}{2}),\quad\varphi(x,0)=\varphi(x,T)\leq
d(x)-c\frac{T^2}{4}\leq-\delta
\ee
uniformly in $x\in\Omega$.

\medskip
(b) There are $t_0$ and $t_1$, with $0<t_0<\frac{T}{2}<t_1<T$, such
that we have \be\label{propertyb}
\min_{x\in\overline{\Omega},t\in[t_0,t_1]}\varphi(x,t)\geq\sigma
\ee
where $0<\sigma<\min_{x\in\overline{\Omega}}d(x)$.

Moreover, if we introduce the space $Q(\sigma)$ that is defined by the following
\be\label{qsigma}
Q{(\sigma)}=\{(x,t)|x\in\Omega, 0\leq t\leq T, \varphi(x,t)\geq\sigma>0\} \ee
Then an important property of $Q(\sigma)$ is that (see \cite{l-t-z}):
\be\label{qsigmaproperty} [t_0,t_1]\times\Omega\subset Q(\sigma)\subset[0,T]\times\Omega\ee
Then for the wave equation of the form
\be\label{carlemaneqn}
w_{tt}(x,t)-\Delta w(x,t)-q(x)w(x,t) = F(x,t), \quad x\in\Omega, t\in[0,T]\ee
we have the following Carleman-type estimate:
\begin{theorem}\label{prop1}
Under the main assumptions (A.1) and (A.2),
with $\varphi(x,t)$ defined in $\eqref{defphi}$. Let $w\in H^{2}(Q)$
be a solution of the equation $\eqref{carlemaneqn}$ where $q\in L^{\infty}(\Omega)$ and
$F\in L_2(Q)$. Then the following one parameter family of estimates
hold true, with $\rho>0$, $\beta>0$, for all $\tau>0$ sufficiently
large and $\epsilon>0$ small:
\begin{multline}\label{carleman}
BT|_w+2\int_Qe^{2\tau\varphi}|F|^2dQ+C_{1,T}e^{2\tau\sigma}\int_Qw^2dQ
\geq (\tau\epsilon\rho-2C_T)\int_Qe^{2\tau\varphi}\left(w_t^2+|\nabla w|^2\right)dQ\\
+\left(2\tau^3\beta+\mathcal{O}(\tau^2)-2C_T\right)\int_{Q{(\sigma)}}e^{2\tau\varphi}w^2dxdt
- c_T\tau^3e^{-2\tau\delta}[E_w(0)+E_w(T)]
\end{multline}
Here $\delta>0$, $\sigma>0$ are the constants in
$\eqref{timesquare}$, $\eqref{propertyb}$, while $C_T$, $c_T$ and
$C_{1,T}$ are positive constants depending on $T$ and $d$. In addition, the
boundary terms $BT|_{w}$ are given explicitly by
\be\label{boundary}
\begin{split}
BT|_{w}& =2\tau\int_0^T\int_{\Gamma_0}e^{2\tau\varphi}(w_t^2-|\nabla w|^2)h\cdot\nu d\Gamma dt \\
            & +8c\tau\int_0^T\int_{\Gamma}e^{2\tau\varphi}(t-\frac{T}{2})w_t\frac{\partial w}{\partial\nu} d\Gamma dt \\
            & +4\tau\int_0^T\int_{\Gamma}e^{2\tau\varphi}(h\cdot\nabla w)\frac{\partial w}{\partial\nu} d\Gamma dt \\
            & +4\tau^2\int_0^T\int_{\Gamma}e^{2\tau\varphi}\left(|h|^2-4c^2(t-\frac{T}{2})^2+\frac{\alpha}{2\tau}\right)w\frac{\partial w}{\partial\nu} d\Gamma dt \\
            & +2\tau\int_0^T\int_{\Gamma_0}e^{2\tau\varphi}\bigg[2\tau^2\left(|h|^2-4c^2(t-\frac{T}{2})^2\right) \\
            & \quad +\tau(\alpha-\Delta d-2c)\bigg]w^2h\cdot\nu d\Gamma dt
\end{split}
\ee
where $\alpha=\Delta d-2c-1+k$ for $0<k<1$ is a constant and $E_w$ is defined as follows:
\be\label{energy}
E_w(t)=\int_{\Omega}[w^2(x,t)+w_t^2(x,t)+|\nabla w(x,t)|^2]d\Omega
\ee
\end{theorem}
An immediate corollary of the estimate is the following (Theorem 6.1 in \cite{l-t-z})
\begin{corollary}
Under the assumptions in Theorem (\ref{prop1}), the following one-parameter family of estimates hold true, for all
$\tau$ sufficiently large, and for any $\epsilon>0$ small:
\be\label{carleman2}
\overline{BT}|_{w}+2\int_0^T\int_{\Omega}e^{2\tau\varphi}F^2dQ+const_{\varphi}\int_0^T\int_{\Omega}F^2dQ\geq k_{\varphi}[E_w(0)+E_w(T)]
\ee
for a constant $k_{\varphi}>0$ while $\overline{BT}|_{w}$ is given by:
\begin{multline}\label{btbar}
\overline{BT}|_{w}=BT|_{w}+const_{\varphi}\left[\int_0^T\int_{\Gamma}\bigg|\frac{\pa w}{\pa\nu}w_t\bigg|d\Gamma dt+\int_{t_0}^{t_1}\int_{\Gamma_0}w^2d\Gamma_0 dt\right. \\
+ \left. \int_0^T\int_{\Gamma_0}|ww_t|d\Gamma_0 dt\right]
\end{multline}
\end{corollary}
\begin{remark}
For the proof of the above Carleman estimate and the corollary, we refer to \cite{l-t-z} and we omit the details here.
\end{remark}

\subsection{Continuous Observability Inequality}

Using the Carleman estimate in last section, we can prove the following observability inequality:
\begin{theorem}\label{theoremobserve}
Under the main assumptions (A.1) and (A.2), for the following initial boundary value problem
\be\label{observe}
\begin{cases}
w_{tt}(x,t) = \Delta w(x,t) + q(x)w(x,t) & \mbox{in } \Omega \times [0,T] \\
w(\cdot,\frac{T}{2}) = w_0(x) & \mbox{in } \Omega \\
w_t(\cdot,\frac{T}{2}) = w_1(x) & \mbox{in } \Omega \\
\frac{\partial w}{\partial \nu}(x,t) = 0 & \mbox{on } \Gamma_1 \times [0,T] \\
\frac{\partial w}{\partial \nu}(x,t) = g(x,t) & \mbox{on } \Gamma_0 \times [0,T]
\end{cases}
\ee
where $w_0\in H^1(\Omega)$, $w_1\in L^2(\Omega)$, $g\in L^2(\Gamma\times[0,T])$ and $q\in L^{\infty}(\Omega)$.
We have the following continuous observability inequality:
\begin{equation*}\label{observeineq}
\|w_0\|^2_{H^1(\Omega)}+\|w_1\|^2_{L^2(\Omega)}\leq C\left(\|w\|^2_{L^2(\Gamma_0\times[0,T])}+\|w_t\|^2_{L^2(\Gamma_0\times[0,T])}+\|g\|^2_{L^2(\Gamma_0\times[0,T])}\right)
\end{equation*}
where $T$ is as in $\eqref{time}$ and $C=C(\Omega,T,\Gamma_0,\varphi,\tau,q)$ is a positive constant.
\end{theorem}
\begin{proof}
For the case when $g=0$, we refer to \cite{l-t-z} where the continuous observability inequality is established for zero Neumann data on
the whole boundary. Here we give the proof for the case of general $g\in L^2(\Gamma_0\times[0,T])$, which is still based on
the proof in \cite{l-t-z}. We first introduce the following result that is from the section 7.2 of \cite{l-t.5}.
\begin{lemma}\label{tangential}
Let $w$ be a solution of the equation
\be\label{wave}
w_{tt}(x,t)=\Delta w(x,t)+q(x)w(x,t)+f(x,t) \ \text{in} \ Q=\Omega\times[0,T]
\ee
with $q\in L^{\infty}(\Omega)$ and $w$ in the following class:
\be
\left\{\begin{aligned}
w\in L^2(0,T;H^1(\Omega))\cap H^1(0,T;L^2(\Omega)) \\
w_t, \frac{\pa w}{\pa\nu}\in L^2(0,T;L^2(\Gamma))
\end{aligned}\right.
\ee
Given $\epsilon>0$, $\epsilon_0>0$ arbitrary, given $T>0$, there exists a constant $C=C(\epsilon,\epsilon_0,T)>0$ such that
\begin{multline}\label{tangential1}
\int_{\epsilon}^{T-\epsilon}\int_{\Gamma}|\nabla_{tan}w|^2d\Gamma dt\leq C\left(\int_0^T\int_{\Gamma}w_t^2+\left(\frac{\pa w}{\pa\nu}\right)^2d\Gamma dt+\|w\|^2_{L^2(0,T;H^{\frac{1}{2}+\epsilon_0}(\Omega))} \right. \\
\left. +\|f\|^2_{H^{\frac{1}{2}+\epsilon_0}(Q)}\right)
\end{multline}
\end{lemma}
Now to prove $\eqref{observeineq}$, we first establish the following weaker conclusion under the assumptions (A.1) and (A.2)
\be\label{polluted}
E\left(\frac{T}{2}\right)\leq C\left(\int_0^T\int_{\Gamma_0}[w^2+w_t^2+g^2]d\Gamma_0dt+\|w\|^2_{L^2(0,T;H^{\frac{1}{2}+\epsilon_0}(\Omega))}\right)
\ee
which is the desired inequality $\eqref{observeineq}$ polluted by the interior lower order term $\|w\|$.
To see this, we introduce a preliminary equivalence first. Let $u\in H^1(\Omega)$, then the following inequality holds true:
there exist positive constants $0<k_1<k_2<\infty$, independent of $u$, such that
\be\label{equivalence}
k_1\int_{\Omega}[u^2+|\nabla u|^2]d\Omega\leq\int_{\Omega}|\nabla u|^2d\Omega+\int_{\tilde{\Gamma_0}}u^2d\Gamma\leq k_2\int_{\Omega}[u^2+|\nabla u|^2]d\Omega
\ee
where $\tilde{\Gamma_0}$ is any (fixed) portion of the boundary $\Gamma$ with positive measure. Inequality $\eqref{equivalence}$ is obtained
by combining the following two inequalities:
\be\label{equivalence1}
\int_{\Omega}u^2 d\Omega\leq c_1\left[\int_{\Omega}|\nabla u|^2d\Omega+\int_{\tilde{\Gamma_0}}u^2d\Gamma\right];\quad
\int_{\tilde{\Gamma_0}}u^2d\Gamma\leq c_2\int_{\Omega}[u^2+|\nabla u|^2]d\Omega
\ee
The inequality on the left of $\eqref{equivalence1}$ replaces Poincar\'{e}'s inequality, while the inequality
on the right of $\eqref{equivalence1}$ stems from (a conservative version of) trace theory. Thus, for $w\in H^2(Q)$, if
we introduce
\be
\varepsilon(t)=\int_{\Omega}\left[|\nabla w(t)|^2+w_t^2(t)\right]d\Omega+\int_{\Gamma_0}w^2(t)d\Gamma_1
\ee
where $\Gamma_0=\Gamma\setminus\Gamma_1$ is as defined in the main assumptions, then $\eqref{equivalence}$ yields the equivalence
\be\label{equivalence2}
a E(t)\leq\varepsilon(t)\leq b E(t)
\ee
for some positive constants $a>0$, $b>0$.

Now in a standard way, we multiply equation $\eqref{wave}$ by $w_t$ and integrate over $\Omega$. After an application of
the first Green's identity, we have
\begin{multline}\label{ibp}
\frac{1}{2}\frac{\pa}{\pa t}\left(\int_{\Omega}[w_t^2+|\nabla w|^2]d\Omega+\int_{\Gamma_0}w^2d\Gamma_0\right)=\int_{\Gamma}\frac{\pa w}{\pa\nu}w_td\Gamma+\int_{\Gamma_0}ww_td\Gamma_0 \\
+\int_{\Omega}\left[q(x)+f\right]w_td\Omega
\end{multline}
Notice that on both sides of (\ref{ibp}) we have added term $\dis\frac{1}{2}\frac{\pa}{\pa t}\int_{\Gamma_0}w^2d\Gamma_0
=\int_{\Gamma_0}ww_td\Gamma_0$. Recalling $\varepsilon(t)$ in $\eqref{equivalence2}$, we integrate (\ref{ibp}) over $(s,t)$ and obtain
\be\label{varepsilonrelation}
\varepsilon(t)=\varepsilon(s)+2\int_s^t\left[\int_{\Gamma}\frac{\pa w}{\pa\nu}w_td\Gamma+\int_{\Gamma_0}ww_td\Gamma_0\right]dr+
2\int_s^t\int_{\Omega}\left[q(x)+f\right]w_td\Omega dr
\ee
We apply Cauthy-Schwartz inequality on $[q(x)+f]w_t$, invoke the left hand side $\dis E(t)\leq\frac{1}{a}\varepsilon(t)$ of
$\eqref{equivalence}$, and obtain
\be\label{varepsilont}
\varepsilon(t)\leq [\varepsilon(s)+N(T)]+C_T\int^t_s\varepsilon(r)dr
\ee
\be\label{varepsilons}
\varepsilon(s)\leq [\varepsilon(t)+N(T)]+C_T\int^t_s\varepsilon(r)dr
\ee
where we have set
\be\label{nt}
N(T)=\int^T_0\int_{\Omega}f^2dQ+2\int^T_0\int_{\Gamma}\bigg|\frac{\pa w}{\pa\nu}w_t\bigg|d\Gamma dt+2\int^T_0\int_{\Gamma_0}|ww_t|d\Gamma_0dt
\ee
Gronwall's inequality applied on $\eqref{varepsilont}$, $\eqref{varepsilons}$ then yields for $0\leq s\leq t\leq T$,
\be\label{gronwall}
\varepsilon(t)\leq [\varepsilon(s)+N(T)]e^{C_T(t-s)}; \quad \varepsilon(s)\leq [\varepsilon(t)+N(T)]e^{C_T(t-s)}
\ee
We consider the following three cases here:

\textbf{Case 1:} $0\leq s\leq t\leq \frac{T}{2}$. In this case we set $t=\frac{T}{2}$ and $s=t$ in the first inequality of $\eqref{gronwall}$;
and set $s=0$ in the second inequality of $\eqref{gronwall}$, to obtain
\be\label{case1}
\varepsilon(\frac{T}{2})\leq[\varepsilon(t)+N(T)]e^{C_T\frac{T}{2}}; \quad \varepsilon(0)\leq[\varepsilon(t)+N(T)]e^{C_T\frac{T}{2}}
\ee
Summing up these two inequalities in $\eqref{case1}$ yields for $0\leq t\leq \frac{T}{2}$,
\be\label{desired1}
\begin{split}
\varepsilon(t)& \geq \frac{\varepsilon(\frac{T}{2})+\varepsilon(0)}{2}e^{-C_T\frac{T}{2}}-N(T) \\
              & \geq \frac{a}{2}[E(\frac{T}{2})+E(0)]e^{-C_T\frac{T}{2}}-N(T)
\end{split}
\ee
after recalling the left hand side of the equivalence in $\eqref{equivalence2}$.

\textbf{Case 2:} $\frac{T}{2}\leq s\leq t\leq T$. In this case we set $t=T$ and $s=t$ in the first inequality of $\eqref{gronwall}$;
and set $s=\frac{T}{2}$ in the second inequality of $\eqref{gronwall}$, to obtain
\be\label{case2}
\varepsilon(T)\leq[\varepsilon(t)+N(T)]e^{C_T\frac{T}{2}}; \quad \varepsilon(\frac{T}{2})\leq[\varepsilon(t)+N(T)]e^{C_T\frac{T}{2}}
\ee
Summing up these two inequalities in $\eqref{case2}$ yields for $\frac{T}{2}\leq t\leq T$,
\be\label{desired2}
\begin{split}
\varepsilon(t)& \geq \frac{\varepsilon(\frac{T}{2})+\varepsilon(T)}{2}e^{-C_T\frac{T}{2}}-N(T) \\
              & \geq \frac{a}{2}[E(\frac{T}{2})+E(T)]e^{-C_T\frac{T}{2}}-N(T)
\end{split}
\ee
after recalling the left hand side of the equivalence in $\eqref{equivalence2}$.

\textbf{Case 3:} $0\leq s\leq\frac{T}{2}\leq t\leq T$. In this case we set $t=0$ and $s=t$ in the first inequality of $\eqref{gronwall}$;
and set $s=\frac{T}{2}$ in the second inequality of $\eqref{gronwall}$, to obtain
\be\label{case3}
\varepsilon(0)\leq[\varepsilon(t)+N(T)]e^{C_T\frac{T}{2}}; \quad \varepsilon(\frac{T}{2})\leq[\varepsilon(t)+N(T)]e^{C_T\frac{T}{2}}
\ee
Summing up these two inequalities in $\eqref{case3}$ yields for $\frac{T}{2}\leq t\leq T$,
\be\label{desired3}
\begin{split}
\varepsilon(t)& \geq \frac{\varepsilon(\frac{T}{2})+\varepsilon(0)}{2}e^{-C_T\frac{T}{2}}-N(T) \\
              & \geq \frac{a}{2}[E(\frac{T}{2})+E(0)]e^{-C_T\frac{T}{2}}-N(T)
\end{split}
\ee
after recalling the left hand side of the equivalence in $\eqref{equivalence2}$.

\medskip
In summary, we get for any $0\leq t\leq T$,
\be\label{desired}
\varepsilon(t)\geq\frac{a}{2}E(\frac{T}{2})e^{-C_T\frac{T}{2}}-N(T)
\ee
We now apply the Corollary 2.2 of the Carleman estimate, except on the interval $[\epsilon, T-\epsilon]$,
rather than on $[0,T]$ as in $\eqref{carleman2}$. Thus, we obtain since $f=0$:
\be\label{corollary1}
\overline{BT}|_{[\epsilon,T-\epsilon]\times\Gamma}\geq k_{\varphi}E(\epsilon)
\ee
where $\overline{BT}|_{[\epsilon,T-\epsilon]\times\Gamma}$ is given as in (\ref{btbar}).
Since we have $\dis\frac{\partial w}{\partial \nu}=0$ on  $\Gamma_1\times[0,T]$ and
$\dis\frac{\partial w}{\partial \nu}=g(x,t)$ on $\Gamma_0\times[0,T]$ by $\eqref{observe}$, with the additional information that
$h\cdot\nu=0$ on $\Gamma_1$ by the assumption (A.1). Thus, by using the explicit expression $\eqref{boundary}$ for $BT|_{w}$, we have that
$\overline{BT}|_{[\epsilon,T-\epsilon]\times\Gamma}$ is given by:
\be\label{btbar1}
\begin{split}
\overline{BT}|_{[\epsilon,T-\epsilon]\times\Gamma}& =2\tau\int_{\epsilon}^{T-\epsilon}\int_{\Gamma_0}e^{2\tau\varphi}(w_t^2-|\nabla w|^2)h\cdot\nu d\Gamma dt \\
            & +8c\tau\int_{\epsilon}^{T-\epsilon}\int_{\Gamma_0}e^{2\tau\varphi}(t-\frac{T}{2})w_tg d\Gamma dt \\
            & +4\tau\int_{\epsilon}^{T-\epsilon}\int_{\Gamma_0}e^{2\tau\varphi}(h\cdot\nabla w)g d\Gamma dt \\
            & +4\tau^2\int_{\epsilon}^{T-\epsilon}\int_{\Gamma_0}e^{2\tau\varphi}\left(|h|^2-4c^2(t-\frac{T}{2})^2+\frac{\alpha}{2\tau}\right)wg d\Gamma dt \\
            & +2\tau\int_{\epsilon}^{T-\epsilon}\int_{\Gamma_0}e^{2\tau\varphi}\left[2\tau^2\left(|h|^2-4c^2(t-\frac{T}{2})^2\right)+\tau(\alpha-\Delta d-2c)\right]w^2h\cdot\nu d\Gamma dt \\
            & +const_{\varphi}\left[\int_{\epsilon}^{T-\epsilon}\int_{\Gamma_0}|gw_t|d\Gamma dt+\int_{t_0}^{t_1}\int_{\Gamma_0}w^2d\Gamma_0 dt+\int_{\epsilon}^{T-\epsilon}\int_{\Gamma_0}|ww_t|d\Gamma_0 dt\right]
\end{split}
\ee Next, by the right side of equivalences $\eqref{equivalence2}$
and $\eqref{desired}$, we obtain \be\label{finishing}
E(\epsilon)\geq\frac{\varepsilon(\epsilon)}{b}\geq\frac{a}{2b}E\left(\frac{T}{2}\right)e^{-C_T\frac{T}{2}}-2\int_0^T\int_{\Gamma}|gw_t|d\Gamma
dt-2\int_0^T\int_{\Gamma_0}|ww_t|d\Gamma_0 dt \ee recalling $N(T)$
in $\eqref{nt}$. We use $\eqref{finishing}$ in $\eqref{corollary1}$.
Finally, we invoke estimate (\ref{tangential1}) of Lemma
\ref{tangential} on the first and the third integral terms of
$\eqref{btbar1}$. This way, we readily obtain $\eqref{polluted}$,
which is our desired inequality polluted by
$\|w\|^2_{L^2(0,T;H^{\frac{1}{2}+\epsilon_0}(\Omega))}$. To
eliminate this interior lower order term, we can apply the standard
compactness/uniqueness argument (e.g.\cite{l-t.1}) by invoking the
global uniqueness Theorem 7.1 in \cite{l-t-z}.
\end{proof}

\section{Proof of Theorem \ref{th1}}

We let $\bar{w}=\bar{w}(f)=w_{t}(f)$ then from $\eqref{linear}$ we have $\bar{w}$, $u$ satisfy
\begin{equation}\label{lineary}
\begin{cases}
\bar{w}_{tt}(x,t) - \Delta \bar{w}(x,t) - q(x)\bar{w}(x,t) = f(x)R_{t}(x,t) & \mbox{in } \Omega \times [0,T] \\
\frac{\pa\bar{w}}{\pa\nu}(x,t) = 0 & \mbox{on } \Gamma_1\times[0,T] \\
\bar{w}(x,t) = -u_{tt}(x,t)-\Delta^2 u(x,t)-\Delta^2 u_t(x,t) & \mbox{on } \Gamma_0\times[0,T] \\
u(x,t)=\frac{\pa u}{\pa \nu}(x,t)=0 & \mbox{on } \pa\Gamma_0\times[0,T] \\
\frac{\pa \bar{w}}{\pa \nu}(x,t)=u_{tt}(x,t) & \mbox{on } \Gamma_0\times[0,T] \\
\bar{w}(\cdot,\frac{T}{2}) = 0 & \mbox{in } \Omega \\
\bar{w}_t(\cdot,\frac{T}{2}) = f(x)R(x,\frac{T}{2}) & \mbox{in } \Omega \\
u(\cdot,\frac{T}{2}) = 0 & \mbox{on } \Gamma_0 \\
u_t(\cdot,\frac{T}{2}) = 0 & \mbox{on } \Gamma_0
\end{cases}
\end{equation}
Under the assumptions in Theorem \ref{th1}, we can apply the Carleman estimate to the wave equation in the system $\eqref{lineary}$
$\bar{w}_{tt}(x,t) - \Delta \bar{w}(x,t) - q(x)\bar{w}(x,t) = f(x)R_{t}(x,t)$ and get
\begin{multline*}\label{carlemany}
BT|_{\bar{w}}+2\int_Qe^{2\tau\varphi}|fR_{t}|^2
dQ+C_{1,T}e^{2\tau\sigma}\int_Q \bar{w}^2dQ
\geq (\tau\epsilon\rho-2C_T)\int_Qe^{2\tau\varphi}[\bar{w}_t^2+|\nabla \bar{w}|^2]dQ \\
+
[2\tau^3\beta+\mathcal{O}(\tau^2)-2C_T]\int_{Q{(\sigma)}}e^{2\tau\varphi}\bar{w}^2dxdt
- c_T\tau^3e^{-2\tau\delta}[E_{\bar{w}}(0)+E_{\bar{w}}(T)]
\end{multline*}
where the boundary terms are given explicitly by \be\label{boundaryy}
\begin{split}
BT|_{\bar{w}}& =2\tau\int_0^T\int_{\Gamma_0}e^{2\tau\varphi}(\bar{w}_t^2-|\nabla \bar{w}|^2)h\cdot\nu d\Gamma dt \\
            & +8c\tau\int_0^T\int_{\Gamma}e^{2\tau\varphi}(t-\frac{T}{2})\bar{w}_t\frac{\partial \bar{w}}{\partial\nu} d\Gamma dt \\
            & +4\tau\int_0^T\int_{\Gamma}e^{2\tau\varphi}(h\cdot\nabla \bar{w})\frac{\partial \bar{w}}{\partial\nu} d\Gamma dt \\
            & +4\tau^2\int_0^T\int_{\Gamma}e^{2\tau\varphi}\left(|h|^2-4c^2(t-\frac{T}{2})^2+\frac{\alpha}{2\tau}\right)\bar{w}\frac{\partial \bar{w}}{\partial\nu} d\Gamma dt \\
            & +2\tau\int_0^T\int_{\Gamma_0}e^{2\tau\varphi}\bigg[2\tau^2\left(|h|^2-4c^2(t-\frac{T}{2})^2\right) \\
            & \quad + \tau(\alpha-\Delta d-2c)\bigg]\bar{w}^2h\cdot\nu d\Gamma dt
\end{split}
\ee
Since we have the extra observation that $u_{tt}(x,t)=0$ on
$\Gamma_0\times[0,T]$ and note that the initial conditions $u(x,\frac{T}{2})=u_t(x,\frac{T}{2})=0$ on $\Gamma_0$, thus by the fundamental theorem of calculus we have
$u(x,t)=0$ on $\Gamma_0\times[0,T]$ and hence from the coupling in the system $\eqref{lineary}$ we get
\be\label{boundaryterm1}
\bar{w}(x,t)=-u_{tt}(x,t)-\Delta^2 u(x,t)-\Delta^2 u_t(x,t)=0 \ \textrm{on} \ \Gamma_0\times[0,T]
\ee
and
\be\label{boundaryterm2}
\frac{\pa\bar{w}}{\pa\nu}(x,t)=u_{tt}(x,t)=0 \ \textrm{on} \ \Gamma_0\times[0,T]
\ee
Plugging $\eqref{boundaryterm1}$ and $\eqref{boundaryterm2}$ into $\eqref{boundaryy}$, note also that $\dis\frac{\pa\bar{w}}{\pa\nu}=0$ on $\Gamma_1\times[0,T]$,
therefore we get $BT|_{\bar{w}}\equiv0$.

In addition, in view of $\eqref{regR}$, $\eqref{crucialR}$, we have $|fR_{t}|\leq C|f|$ for some positive constant $C$ depend on $R_t$. Moreover,
notice that $\dis\lim_{\tau\to\infty}\tau^3e^{-2\tau\delta}=0$. Hence when $\tau$ is sufficiently large, the above Carleman estimate can be rewritten as the following:
\begin{equation}\label{ineq2}
C_{1,\tau}\int_Qe^{2\tau\varphi}[\bar{w}_t^2+|\nabla \bar{w}|^2]dQ+C_{2,\tau}\int_{Q(\sigma)}e^{2\tau\varphi}\bar{w}^2dxdt
\leq C\int_Qe^{2\tau\varphi}|f|^2dQ+Ce^{2\tau\sigma}
\end{equation}
where we set \be\label{ctau} C_{1,\tau}=\tau\epsilon\rho-2C_T, \quad
C_{2,\tau}=2\tau^3\beta+\mathcal{O}(\tau^2)-2C_T\ee
and $C$ denote generic constants which do not depend on $\tau$ and henceforth we will use this notation for the rest of this paper.
In addition, note that $f$ is time-independent, so if we differentiate the system $\eqref{lineary}$ in time twice, we can get
the following wave equations for $\bar{w_t}$ and $\bar{w_{tt}}$:
\be\label{eqnfrtt}
(\bar{w_t})_{tt}(x,t) - \Delta \bar{w_t}(x,t) - q(x)\bar{w_t}(x,t) = f(x)R_{tt}(x,t)
\ee
and
\be\label{eqnfrtt}
(\bar{w_{tt}})_{tt}(x,t) - \Delta \bar{w_{tt}}(x,t) - q(x)\bar{w_{tt}}(x,t) = f(x)R_{ttt}(x,t)
\ee
Notice the assumptions $\eqref{regR}$, $\eqref{h2reg}$, therefore we have similarly as $\eqref{ineq2}$ the following estimates for the
two new systems:
\begin{multline}\label{ineq3}
C_{1,\tau}\int_Qe^{2\tau\varphi}[\bar{w}_{tt}^2+|\nabla \bar{w}_t|^2]dQ+C_{2,\tau}\int_{Q(\sigma)}e^{2\tau\varphi}\bar{w}_t^2dxdt
\leq C\int_Qe^{2\tau\varphi}|f|^2dQ+Ce^{2\tau\sigma}
\end{multline}
and
\begin{multline}\label{ineq4}
C_{1,\tau}\int_Qe^{2\tau\varphi}[\bar{w}_{ttt}^2+|\nabla \bar{w}_{tt}|^2]dQ+C_{2,\tau}\int_{Q(\sigma)}e^{2\tau\varphi}\bar{w}_{tt}^2dxdt
\leq C\int_Qe^{2\tau\varphi}|f|^2dQ+Ce^{2\tau\sigma}
\end{multline}
where $\tau$ is sufficiently large and $C_{1,\tau}$, $C_{2,\tau}$ are defined as in $\eqref{ctau}$.

Adding $(\ref{ineq2})$, $(\ref{ineq3})$ and $(\ref{ineq4})$ together we then have
\begin{multline}\label{ineq5}
C_{1,\tau}\int_Qe^{2\tau\varphi}[\bar{w}_t^2+\bar{w}_{tt}^2+\bar{w}_{ttt}^2+|\nabla
\bar{w}|^2+|\nabla \bar{w}_t|^2+|\nabla \bar{w}_{tt}|^2]dQ \\+
C_{2,\tau}\int_{Q(\sigma)}e^{2\tau\varphi}[\bar{w}^2+\bar{w}_t^2+\bar{w}_{tt}^2]dxdt
\leq C\left(\int_Qe^{2\tau\varphi}|f|^2 dQ+e^{2\tau\sigma}\right)
\end{multline}
Again we use the wave equation $\bar{w}_{tt}(x,t)-\Delta \bar{w}(x,t)-q(x)\bar{w}(x,t)=f(x)R_{t}(x,t)$,
plugging in the initial time of $t=\frac{T}{2}$ and use the zero initial conditions of $\bar{w}(\cdot,\frac{T}{2})=0$,
we have
\be\label{initial}
\bar{w}_{tt}(x,\frac{T}{2})-\Delta\bar{w}(x,\frac{T}{2})-q(x)\bar{w}(x,\frac{T}{2})=\bar{w}_{tt}(x,\frac{T}{2})=f(x)R_t(x,\frac{T}{2})\ee
Since $|R_t(x,\frac{T}{2})|\geq r_1>0$ from $\eqref{crucialR}$,
therefore we have $|f(x)|\leq C|\bar{w}_{tt}(x,\frac{T}{2})|$
and hence we have the following estimates on
$\dis\int_Qe^{2\tau\varphi}|f|^2dQ$: \be\label{mainineq}
\begin{split}
\int_Qe^{2\tau\varphi}|f|^2dQ& =\int_0^T\int_{\Omega}e^{2\tau\varphi(x,t)}|f(x)|^2d\Omega dt \\
                             & \leq C\int_0^T\int_{\Omega}e^{2\tau\varphi(x,t)}|\bar{w}_{tt}(x,\frac{T}{2})|^2d\Omega dt \\
                             & \leq C\int_{\Omega}e^{2\tau\varphi(x,\frac{T}{2})}|\bar{w}_{tt}(x,\frac{T}{2})|^2d\Omega \\
                             & = C\left(\int_{\Omega}\int_0^{\frac{T}{2}} \frac{d}{ds}(e^{2\tau\varphi(x,s)}|\bar{w}_{tt}(x,s)|^2)ds d\Omega+\int_{\Omega}e^{2\tau\varphi(x,0)}|\bar{w}_{tt}(x,0)|^2d\Omega \right) \\
                             & = C\left(4c\tau\int_{\Omega}\int_0^{\frac{T}{2}}(\frac{T}{2}-s)e^{2\tau\varphi(x,s)}|\bar{w}_{tt}(x,s)|^2dsd\Omega \right. \\
                             & \quad + \left.2\int_{\Omega}\int_0^{\frac{T}{2}}e^{2\tau\varphi}|\bar{w}_{tt}(x,s)||\bar{w}_{ttt}(x,s)|dsd\Omega+\int_{\Omega}e^{2\tau\varphi(x,0)}|\bar{w}_{tt}(x,0)|^2d\Omega\right) \\
                             & \leq C\left(\tau\int_{\Omega}\int^{\frac{T}{2}}_0e^{2\tau\varphi}|\bar{w}_{tt}|^2 dt d\Omega
                                       + \int_{\Omega}\int^{\frac{T}{2}}_0 e^{2\tau\varphi}(|\bar{w}_{tt}|^2+|\bar{w}_{ttt}|)^2dt d\Omega \right. \\
                             & \quad + \left.\int_{\Omega}|\bar{w}_{tt}(x,0)|^2d\Omega \right) \\
                             & \leq C\left(\tau\int_Q e^{2\tau\varphi}|\bar{w}_{tt}|^2dQ+\int_Qe^{2\tau\varphi}(|\bar{w}_{tt}|^2+|\bar{w}_{ttt}|^2)dQ\right) \\
                             & = C\left((\tau+1)\int_Q e^{2\tau\varphi}|\bar{w}_{tt}|^2dQ+\int_Qe^{2\tau\varphi}|\bar{w}_{ttt}|^2dQ\right)
\end{split}
\ee
where in the above estimates we use the definition $\eqref{defphi}$ and the property
$\eqref{propertya}$ of $\varphi$ as well as Cauthy-Schwartz inequality.
Collecting $\eqref{mainineq}$ with $(\ref{ineq5})$, we have
\begin{multline}\label{ineq6}
C_{1,\tau}\int_Qe^{2\tau\varphi}[\bar{w}_t^2+\bar{w}_{tt}^2+\bar{w}_{ttt}^2+|\nabla
\bar{w}|^2+|\nabla \bar{w}_t|^2+|\nabla \bar{w}_{tt}|^2]dQ \\+
C_{2,\tau}\int_{Q(\sigma)}e^{2\tau\varphi}[\bar{w}^2+\bar{w}_t^2+\bar{w}_{tt}^2]dxdt
\leq C\left((\tau+1)\int_Q e^{2\tau\varphi}|\bar{w}_{tt}|^2dQ+\int_Qe^{2\tau\varphi}|\bar{w}_{ttt}|^2dQ+e^{2\tau\sigma}\right)
\end{multline}
Note that in $(\ref{ineq6})$, the right hand side term
$C\int_Qe^{2\tau\varphi}|\bar{w}_{ttt}|^2dQ$ can be absorbed by the term $C_{1,\tau}\int_Qe^{2\tau\varphi}[\bar{w}_t^2+\bar{w}_{tt}^2+\bar{w}_{ttt}^2]dQ$
on the left hand side when $\tau$ is large enough. In addition, since $e^{2\tau\varphi}<e^{2\tau\sigma}$ on $Q\setminus Q(\sigma)$ by the definition of $Q(\sigma)$, we have
\be\label{absorb}
\begin{split}
C(\tau+1)\int_Q e^{2\tau\varphi}|\bar{w}_{tt}|^2dQ
& =C(\tau+1)\left(\int_{Q(\sigma)}e^{2\tau\varphi}|\bar{w}_{tt}|^2dtdx+\int_{Q\setminus Q(\sigma)}e^{2\tau\varphi}|\bar{w}_{tt}|^2dxdt\right) \\
& \leq C(\tau+1)\left(\int_{Q(\sigma)}e^{2\tau\varphi}|\bar{w}_{tt}|^2dtdx+e^{2\tau\sigma}\int_{Q\setminus Q(\sigma)}|\bar{w}_{tt}|^2dxdt\right) \\
& \leq C(\tau+1)\int_{Q(\sigma)}e^{2\tau\varphi}|\bar{w}_{tt}|^2dtdx+C(\tau+1)e^{2\tau\sigma}
\end{split}
\ee
Again $C(\tau+1)\int_{Q(\sigma)}e^{2\tau\varphi}|\bar{w}_{tt}|^2dtdx$ on
the right hand side of $\eqref{absorb}$ can be absorbed by $C_{2,\tau}\int_{Q(\sigma)}e^{2\tau\varphi}[\bar{w}^2+\bar{w}_t^2+\bar{w}_{tt}^2]dxdt$
on the left hand side of $(\ref{ineq6})$ when taking $\tau$ large enough. Therefore $(\ref{ineq6})$ becomes to
\begin{multline}\label{ineq7}
C_{1,\tau}^{'}\int_Qe^{2\tau\varphi}[\bar{w}_t^2+\bar{w}_{tt}^2+\bar{w}_{ttt}^2+|\nabla
\bar{w}|^2+|\nabla \bar{w}_t|^2+|\nabla \bar{w}_{tt}|^2]dQ \\+
C_{2,\tau}^{'}\int_{Q(\sigma)}e^{2\tau\varphi}[\bar{w}^2+\bar{w}_t^2+\bar{w}_{tt}^2]dxdt
\leq C\left((\tau+1)e^{2\tau\sigma}+e^{2\tau\sigma}+\tau^3e^{-2\tau\delta}\right)
\end{multline}
Where we have \be\label{ctauprime} C_{1,\tau}^{'}=\tau\epsilon\rho-C, \quad
C_{2,\tau}^{'}=2\tau^3\beta+\mathcal{O}(\tau^2) \ee
Now we take $\tau$ sufficiently large such that $C_{1,\tau}^{'}>0$, $C_{2,\tau}^{'}>0$. Then in $(\ref{ineq7})$ we can drop
the first term on the left hand side and get \be\label{ineq8}
\begin{split}
C_{2,\tau}^{'}\int_{Q(\sigma)}e^{2\tau\varphi}[\bar{w}^2+\bar{w}_t^2+\bar{w}_{tt}^2]dxdt& \leq C[(\tau+1)e^{2\tau\sigma}+e^{2\tau\sigma}] \\
                                                                      & \leq C(\tau+2)e^{2\tau\sigma}
\end{split}
\ee
Note again from $\eqref{qsigma}$ the definition of $Q(\sigma)$, we
have $e^{2\tau\varphi}\geq e^{2\tau\sigma}$ on $Q(\sigma)$, therefore
$(\ref{ineq8})$ implies
\be\label{ineq9}
C_{2,\tau}^{'}\int_{Q(\sigma)}[\bar{w}^2+\bar{w}_t^2+\bar{w}_{tt}^2]dxdt
\leq C(\tau+2) \ee
Divide $\tau+2$ on both sides of $\eqref{ineq9}$,
we get \be\label{ineq10}
\frac{C_{2,\tau}^{'}}{\tau+2}\int_{Q(\sigma)}[\bar{w}^2+\bar{w}_t^2+\bar{w}_{tt}^2]dxdt
\leq C \ee
By $\eqref{ctauprime}$, $\frac{C_{2,\tau}^{'}}{\tau+2}\to\infty$ as
$\dis\tau\to\infty$, thus $\eqref{ineq10}$ implies that we must have
$\bar{w}\equiv0$ on $Q(\sigma)$ and hence we have
\be\label{identity}
f(x)R_t(x,t)=\bar{w}_{tt}(x,t)-\Delta \bar{w}(x,t)-q(x)\bar{w}(x,t)=0, \quad (x,t)\in Q(\sigma) \ee
Recall again that $|R_t(x,\frac{T}{2})|\geq r_1>0$ from$\eqref{crucialR}$ and the property that
$Q\supset Q(\sigma)\supset[t_0,t_1]\times\Omega$ from
$\eqref{qsigmaproperty}$. Thus we have from $\eqref{identity}$ that
$f(x)\equiv0$, for all $x\in\Omega$. $\qquad\Box$

\section{Proof of Theorem \ref{th2}}

Setting $f(x)=q(x)-p(x)$, $w(x,t)=z(q)(x,t)-z(p)(x,t)$, $u(x,t)=v(q)(x,t)-v(p)(x,t)$ and $R(x,t)=z(p)(x,t)$,
we then obtain $\eqref{linear}$ after the subtraction of
$\eqref{nonlinear}$ with $p$ from $\eqref{nonlinear}$ with $q$. Since $R(x,\frac{T}{2})=z(p)(x,\frac{T}{2})=z_0(x)$ and
$R_t(x,\frac{T}{2})=z_t(p)(x,\frac{T}{2})=z_1(x)$, the conditions $\eqref{crucialz}$ imply $\eqref{crucialR}$.
In addition, the condition $v(q)(x,t)=v(p)(x,t)$, $x\in\Gamma_0$, $t\in[0,T]$ implies that $u(x,t)=0$ on $\Gamma_0\times[0,T]$ and
$\eqref{differenceh2}$ implies $\eqref{h2reg}$.
Therefore from the above Theorem \ref{th1} we conclude $f(x)=q(x)-p(x)=0$, i.e., $q(x)=p(x)$, $x\in\Omega$. \hfill $\Box$

\section{Proof of Theorem \ref{th3}}

In relation with this system $\eqref{lineary}$, we define $\psi$ which satisfies the following equation
\begin{equation}\label{eqpsi}
\begin{cases}
\psi_{tt}(x,t) = \Delta \psi(x,t) + q(x)\psi(x,t) & \mbox{in } \Omega \times [0,T] \\
\frac{\pa\psi}{\pa\nu}(x,t) = 0 & \mbox{on } \Gamma_1\times[0,T] \\
\frac{\pa \psi}{\pa \nu}(x,t) = u_{tt}(x,t) & \mbox{on } \Gamma_0\times[0,T] \\
\psi(\cdot,\frac{T}{2}) = 0 & \mbox{in } \Omega \\
\psi_t(\cdot,\frac{T}{2}) = f(x)R(x,\frac{T}{2}) & \mbox{in } \Omega
\end{cases}
\end{equation}
Set  $y=\bar{w}-\psi$, then we have $y$ satisfies the following initial-boundary value problem
\begin{equation}\label{eqy}
\begin{cases}
y_{tt}(x,t) - \Delta y(x,t) - q(x)y(x,t) = f(x)R_t(x,t) & \mbox{in } \Omega \times [0,T] \\
\frac{\pa y}{\pa\nu}(x,t) = 0 & \mbox{on } \Gamma\times[0,T] \\
y(\cdot,\frac{T}{2}) = 0 & \mbox{in } \Omega \\
y_t(\cdot,\frac{T}{2}) = 0 & \mbox{in } \Omega \\
\end{cases}
\end{equation}
It is easy to see that both $\eqref{eqpsi}$ and $\eqref{eqy}$ are well-posed.
For the system $\eqref{eqpsi}$, we apply the continuous observability inequality in Theorem \ref{theoremobserve} to get
\be\label{ineqfr}
\|fR(\cdot,\frac{T}{2})\|^2_{L^2(\Omega)}\leq C\left(\|\psi\|^2_{L^2(\Gamma_0\times[0,T])}+\|\psi_t\|^2_{L^2(\Gamma_0\times[0,T])}+\|\frac{\pa\psi}{\pa\nu}\|^2_{L^2(\Gamma_0\times[0,T])}\right)\ee
Notice that $|R(x,\frac{T}{2})|\geq r_0>0$, $\frac{\pa \psi}{\pa \nu}(x,t) = u_{tt}(x,t)$ on $\Gamma_0\times[0,T]$ and
$\frac{\pa \psi}{\pa \nu}(x,t) = 0$ on $\Gamma_1\times[0,T]$, therefore we have from $\eqref{ineqfr}$
\be\label{ineqfr1}
\|f\|_{L^2(\Omega)}\leq C\left(\|\psi\|_{L^2(\Gamma_0\times[0,T])}+\|\psi_t\|_{L^2(\Gamma_0\times[0,T])}+\|u_{tt}\|_{L^2(\Gamma_0\times[0,T])}\right)\ee
On the other hand, for the system $\eqref{eqy}$, we have the following lemma:
\begin{lemma}
Let $q\in L^{\infty}(\Omega)$ and $R(x,t)$ satisfies $R_t\in H^{\frac{1}{2}+\epsilon}(0,T;L^{\infty}(\Omega))$
for some $0<\epsilon<\frac{1}{2}$ as in Theorem \ref{th3}. If we define the operators $K$ and $K_1$ by
$K, K_1: L^2(\Omega)\rightarrow L^2(\Gamma_0\times[0,T])$, such that
\be\label{operator}
(Kf)(x,t)=y(x,t), \quad (K_1f)(x,t)=y_t(x,t), \quad x\in\Gamma_0,t\in[0,T]\ee
where $y$ is the unique solution of the equation $\eqref{eqy}$. Then $K$ and $K_1$  are both compact operators.
\end{lemma}
\begin{proof}
It suffices to just show that $K_1$ is compact, then it follows similarly that $K$ is also compact.
Since $f\in L^2(\Omega)$ and $R_t\in H^{\frac{1}{2}+\epsilon}(0,T;L^{\infty}(\Omega))$, we have
\be\label{regfRt}
fR_t\in H^{\frac{1}{2}+\epsilon}(0,T;L^2(\Omega))\ee
Therefore we have the solution $y$ satisfies (e.g. Corollary 5.3 in \cite{l-t.4})
\be\label{sharpy}
y\in C([0,T];H^{\frac{3}{2}+\epsilon}(\Omega)), \quad y_t\in C([0,T];H^{\frac{1}{2}+\epsilon}(\Omega))
\ee
Hence by $\eqref{regfRt}$, $q\in L^{\infty}(\Omega)$ and $y_{tt}=\Delta y+q(x)y+fR_t$ we can get
\be\label{regytt}
y_{tt}\in L^2(0,T;H^{-\frac{1}{2}+\epsilon}(\Omega))
\ee
In addition, by $\eqref{sharpy}$ and trace theorem we have $y_t\in C([0,T];H^{\epsilon}(\Gamma))$.
Since the embedding $H^{\epsilon}(\Gamma)\to L^2(\Gamma)$ is compact, we have by Lions-Aubin's compactness criterion
(e.g. Proposition III.1.3 in \cite{s}) that the operator $K_1$ is a compact operator.
\end{proof}
Now we have that the inequality $\eqref{ineqfr1}$ becomes to
\be\label{ineq11}
\begin{split}
\|f\|_{L^2(\Omega)} & \leq C\left(\|\psi\|_{L^2(\Gamma_0\times[0,T])}+\|\psi_t\|_{L^2(\Gamma_0\times[0,T])}+\|u_{tt}\|_{L^2(\Gamma_0\times[0,T])}\right) \\
                    & \leq C\left(\|\bar{w}-y\|_{L^2(\Gamma_0\times[0,T])}+\|\bar{w}_t-y_t\|_{L^2(\Gamma_0\times[0,T])}+\|u_{tt}\|_{L^2(\Gamma_0\times[0,T])}\right) \\
                    & \leq C\left(\|\bar{w}\|_{L^2(\Gamma_0\times[0,T])}+\|\bar{w}_t\|_{L^2(\Gamma_0\times[0,T])}+\|u_{tt}\|_{L^2(\Gamma_0\times[0,T])}\right) \\
 & \qquad +C\|y\|_{L^2(\Gamma_0\times[0,T])}+C\|y_t\|_{L^2(\Gamma_0\times[0,T])} \\
                    & = C\left(\|\bar{w}\|_{L^2(\Gamma_0\times[0,T])}+\|\bar{w}_t\|_{L^2(\Gamma_0\times[0,T])}+\|u_{tt}\|_{L^2(\Gamma_0\times[0,T])}\right) \\
 & \qquad +C\|Kf\|_{L^2(\Gamma_0\times[0,T])}+C\|K_1f\|_{L^2(\Gamma_0\times[0,T])} \\
                   & \leq C\left(\|u_{tt}\|_{L^2(\Gamma_0\times[0,T])}+\|u_{ttt}\|_{L^2(\Gamma_0\times[0,T])}+\|\Delta^2u_{tt}\|_{L^2(\Gamma_0\times[0,T])}\right) \\
& \qquad +C\|Kf\|_{L^2(\Gamma_0\times[0,T])}+C\|K_1f\|_{L^2(\Gamma_0\times[0,T])}
\end{split}
\ee
where in the last step we use the
coupling $\bar{w}(x,t) = -u_{tt}(x,t)-\Delta^2 u(x,t)-\Delta^2
u_t(x,t)$ on $\Gamma_0\times[0,T]$ from $\eqref{lineary}$ and again
the initial conditions $u(\cdot,\frac{T}{2})=u_t(\cdot,\frac{T}{2})=0$ on $\Gamma_0\times[0,T]$ so that by the
fundamental theorem of calculus, we have
\be\label{poincare}
\|u\|_{L^2(\Gamma_0\times[0,T])}\leq C\|u_t\|_{L^2(\Gamma_0\times[0,T])}\leq C\|u_{tt}\|_{L^2(\Gamma_0\times[0,T])}
\ee
To complete the proof, we need to absorb the last two terms in $\eqref{ineq11}$. To achieve that, we apply the
compactness-uniqueness argument. For simplicity we denote
\begin{equation*}
\|u\|_X=\|u_{tt}\|_{L^2(\Gamma_0\times[0,T])}+\|u_{ttt}\|_{L^2(\Gamma_0\times[0,T])}
+\|\Delta^2u_{tt}\|_{L^2(\Gamma_0\times[0,T])}
\end{equation*}
Suppose contrarily that the inequality $\eqref{stability}$
does not hold. Then there exists $f_n\in L^2(\Omega)$, $n\geq 1$
such that \be\label{contrary1} \|f_n\|_{L^2(\Omega)}=1, \quad n\geq
1 \ee and \be\label{contrary2}
\lim_{n\to\infty}\|u(f_n)\|_{X}=0
\ee
From $\eqref{contrary1}$, there exists a subsequence, denoted
again by $\{f_n\}_{n\geq1}$ such that $f_n$ converges to some
$f_0\in L^2(\Omega)$ weakly in $L^2(\Omega)$. Moreover, since $K$ and $K_1$ are compact, we have
\be\label{compactidentity}
\lim_{m,n\to\infty}\|Kf_n-Kf_m\|_{L^2(\Gamma_0\times[0,T])}=0, \quad \lim_{m,n\to\infty}\|K_1f_n-K_1f_m\|_{L^2(\Gamma_0\times[0,T])}=0
\ee
On the other hand, it follows from $\eqref{ineq11}$ that
\be\label{cauthyf}
\begin{split}
\|f_n-f_m\|_{L^2(\Omega)}& \leq C\|u(f_n)-u(f_m)\|_{X}+C\|Kf_n-Kf_m\|_{L^2(\Gamma_0\times[0,T])} \\
& \qquad +C\|K_1f_n-K_1f_m\|_{L^2(\Gamma_0\times[0,T])} \\
& \leq C\|u(f_n)\|_{X}+C\|u(f_m)\|_{X}+C\|Kf_n-Kf_m\|_{L^2(\Gamma_0\times[0,T])} \\
& \qquad +C\|K_1f_n-K_1f_m\|_{L^2(\Gamma_0\times[0,T])}
\end{split}
\ee
Thus by $\eqref{contrary2}$ and $\eqref{compactidentity}$,
we have that \be\label{iden1}
\lim_{m,n\to\infty}\|f_n-f_m\|_{L^2(\Omega)}=0 \ee
and hence $f_n$ converges strongly to $f_0$ in $L^2(\Omega)$.
So by $\eqref{contrary1}$ we obtain \be\label{iden3}
\|f_0\|_{L^2(\Omega)}=1 \ee
On the other hand, by $\eqref{crucialR}$
and a usual a-priori estimate, we have that \be
\begin{split}
\|\bar{w}(f)\|_{C([0,T];H^1(\Omega))}+\|\bar{w}_t(f)\|_{C([0,T];L^2(\Omega))} & \leq C\|fR_t\|_{L^1(0,T;L^2(\Omega))} \\
                                                                              & \leq C\|R_t\|_{L^1(0,T;L^{\infty}(\Omega))}\|f\|_{L^2(\Omega)}
\end{split}
\ee
Hence trace theorem implies that \be\label{traceineq}
\|\bar{w}(f)\|_{L^2(\Gamma_0\times[0,T])}\leq C\|f\|_{L^2(\Omega)} \ee
where $C>0$ depends on $\|R_t\|_{L^1(0,T;L^{\infty}(\Omega))}$.
Therefore by $\eqref{traceineq}$ we have \be\label{ineq12}
\lim_{n\to\infty}\|\bar{w}(f_n)-\bar{w}(f_0)\|_{L^2(\Gamma_0\times[0,T])}\leq
C\lim_{n\to\infty}\|f_n-f_0\|_{L^2(\Omega)}=0 \ee
Moreover, by $\eqref{contrary2}$ and the coupling $\bar{w}(x,t) = -u_{tt}(x,t)-\Delta^2 u(x,t)-\Delta^2 u_t(x,t)$
on $\Gamma_0\times[0,T]$, we have
\be\label{iden4}
\lim_{n\to\infty}\|\bar{w}(f_n)\|_{L^2(\Gamma_0\times[0,T])}\leq \lim_{n\to\infty}\|u\|_X=0 \ee
Thus by $\eqref{ineq12}$ and $\eqref{iden4}$, we obtain
\be\label{last} \bar{w}(f_0)(x,t)=0, \quad x\in\Gamma_0, t\in[0,T] \ee
Therefore from $\eqref{lineary}$ we have $u=u(f_0)$ satisfies the initial boundary problem:
\begin{equation}\label{equationu}
\begin{cases}
-u_{tt}(x,t)-\Delta^2 u(x,t)-\Delta^2 u_t(x,t) = 0 & \mbox{in } \Gamma_0\times[0,T] \\
u(x,t)=\frac{\pa u}{\pa \nu}(x,t)=0 & \mbox{on } \pa\Gamma_0\times[0,T] \\
u(\cdot,\frac{T}{2}) = 0 & \mbox{in } \Gamma_0 \\
u_t(\cdot,\frac{T}{2}) = 0 & \mbox{in } \Gamma_0
\end{cases}
\end{equation}
which has only zero solution, namely, we have
$u(f_0)(x,t)=0, x\in\Gamma_0, t\in [0,T]$. Therefore by the
uniqueness theorem \ref{th1}, we have
$f_0\equiv0$ in $\Omega$ which contradicts with $\eqref{iden3}$. Thus we must have
\begin{equation}
\|f\|_{L^2(\Omega)}\leq C\left(\|u_{tt}\|_{L^2(\Gamma_0\times[0,T])}+\|u_{ttt}\|_{L^2(\Gamma_0\times[0,T])}+\|\Delta^2u_{tt}\|_{L^2(\Gamma_0\times[0,T])}\right)
\end{equation}
and the proof of the theorem is complete. \hfill $\Box$

\section{Proof of Theorem \ref{th4}}

We now go back to the original system $\eqref{nonlinear}$.

\textbf{Case 1: n=2}.
Let $\bar{z}=z_t$ and $\bar{v}=v_t$, then the system $\eqref{nonlinear}$ becomes to
\begin{equation}\label{nonlinear2}
\begin{cases}
\bar{z}_{tt}(x,t) = \Delta \bar{z}(x,t) + q(x)\bar{z}(x,t) & \mbox{in } \Omega \times [0,T] \\
\frac{\pa \bar{z}}{\pa\nu}(x,t) = 0 & \mbox{on } \Gamma_1\times[0,T] \\
\bar{z}_t(x,t) = -\bar{v}_{tt}(x,t)-\Delta^2 \bar{v}(x,t)-\Delta^2 \bar{v}_t(x,t) & \mbox{on } \Gamma_0\times[0,T] \\
\bar{v}(x,t)=\frac{\pa \bar{v}}{\pa \nu}(x,t)=0 & \mbox{on } \pa\Gamma_0\times[0,T] \\
\frac{\pa \bar{z}}{\pa \nu}(x,t)=\bar{v}_t(x,t) & \mbox{on } \Gamma_0\times[0,T] \\
\bar{z}(\cdot,\frac{T}{2}) = z_1(x) & \mbox{in } \Omega \\
\bar{z}_t(\cdot,\frac{T}{2}) = \Delta z_0(x)+q(x)z_0 & \mbox{in } \Omega \\
\bar{v}(\cdot,\frac{T}{2}) = v_1(x) & \mbox{on } \Gamma_0 \\
\bar{v}_t(\cdot,\frac{T}{2}) = -z_1(x)-\Delta^2 v_0(x)-\Delta^2 v_1(x) & \mbox{on } \Gamma_0
\end{cases}
\end{equation}
By using the similar operator setting as in section 1.2 and notice the new initial conditions, we can compute the domain of the operator $\mathcal{A}^2$:
\begin{equation}\label{DomA2}
\begin{split}
D(\mathcal{A}^2)& =\{[z_0,z_1,v_0,v_1]^T:(z_1,(-A_N+q)z_0+Bv_1,v_1,-\textbf{\AA}(v_0+v_1)-B^{*}z_1\in D(\mathcal{A})\} \\
                & =\{[z_0,z_1,v_0,v_1]^T: z_1\in H^2_{\Gamma_1}(\Omega), (-A_N+q)z_0+Bv_1\in H^1_{\Gamma_1}(\Omega), v_0\in H^2_0(\Gamma_0), \\
& \qquad v_1\in H^2_0(\Gamma_0), \textbf{\AA}(v_0+v_1)+B^{*}z_1\in H^2_0(\Gamma_0), \\
& \qquad v_1-\textbf{\AA}(v_0+v_1)-B^{*}z_1\in D(\textbf{\AA}), \frac{\pa z_1}{\pa\nu}|_{\Gamma_0}=-\textbf{\AA}(v_0+v_1)-B^{*}z_1\} \\
                & =\{[z_0,z_1,v_0,v_1]^T: z_1\in H^2_{\Gamma_1}(\Omega),(\Delta+q)z_0\in H^1_{\Gamma_1}(\Omega), \frac{\pa z_0}{\pa\nu}|_{\Gamma_0}=v_1, \\
& \qquad v_0\in H^2_0(\Gamma_0), v_1\in H^2_0(\Gamma_0),\textbf{\AA}(v_0+v_1)+B^{*}z_1\in H^2_0(\Gamma_0), \\
& \qquad v_1-\textbf{\AA}(v_0+v_1)-B^{*}z_1\in D(\textbf{\AA}), \frac{\pa z_1}{\pa\nu}|_{\Gamma_0}=-\textbf{\AA}(v_0+v_1)-B^{*}z_1\} \\
                & =\{[z_0,z_1,v_0,v_1]^T: z_0\in H^3_{\Gamma_1}(\Omega),z_1\in H^2_{\Gamma_1}(\Omega),v_0\in H^2_0(\Gamma_0), v_1\in H^2_0(\Gamma_0), \\
& \qquad\textbf{\AA}(v_0+v_1)+B^{*}z_1\in H^2_0(\Gamma_0), v_1-\textbf{\AA}(v_0+v_1)-B^{*}z_1\in D(\textbf{\AA}), \\
& \qquad\frac{\pa z_0}{\pa\nu}|_{\Gamma_0}=v_1, \frac{\pa z_1}{\pa\nu}|_{\Gamma_0}=-\textbf{\AA}(v_0+v_1)-B^{*}z_1\}
\end{split}
\end{equation}
where in the last step $z_0\in H^3_{\Gamma_1}(\Omega)$ is from elliptic theory when provided that $q(x)\in W^{1,\infty}(\Omega)$.
Therefore when $z_0\in H^3_{\Gamma_1}(\Omega)$, $z_1\in H^2_{\Gamma_1}(\Omega)$, $v_0\in H^2_0(\Gamma_0)$, $v_1\in H^2_0(\Gamma_0)$
with compatible conditions as in $D(\mathcal{A}^2)$ and $q\in W^{1,\infty}(\Omega)$, then from semigroup theory we have that the solution of $\eqref{nonlinear2}$ satisfies
\be\label{barztztt}
\bar{z}_t\in C([0,T];H^1(\Omega)), \quad \bar{z}_{tt}\in C([0,T];L^2(\Omega))
\ee
Hence we have on the one hand
\be\label{regzt11}
z_t\in H^1(0,T;H^1(\Omega))
\ee
On the other hand, from $\eqref{barztztt}$ and $\bar{z}_{tt}(x,t) = \Delta \bar{z}(x,t) + q(x)\bar{z}(x,t)$, we have by elliptic theory that
\be\label{regzt02}
z_t=\bar{z}\in L^2(0,T;H^2(\Omega))
\ee
Interpolate between $\eqref{regzt11}$ and $\eqref{regzt02}$, we have for $0<\epsilon<\frac{1}{2}$,
\be
z_t\in H^{\frac{1}{2}+\epsilon}(0,T;H^{\frac{3}{2}-\epsilon}(\Omega))\subset H^{\frac{1}{2}+\epsilon}(0,T;L^{\infty}(\Omega))
\ee
where the inclusion is by Sobolev embedding theorem.

\textbf{Case 2: n=3}.
We let $\bar{\bar{z}}=z_{tt}$, $\bar{\bar{v}}=v_{tt}$, then we have $\bar{\bar{z}}$, $\bar{\bar{v}}$ satisfy
\begin{equation}\label{nonlinear3}
\begin{cases}
\bar{\bar{z}}_{tt}(x,t) = \Delta \bar{\bar{z}}(x,t) + q(x)\bar{\bar{z}}(x,t) & \mbox{in } \Omega \times [0,T] \\
\frac{\pa \bar{\bar{z}}}{\pa\nu}(x,t) = 0 & \mbox{on } \Gamma_1\times[0,T] \\
\bar{\bar{z}}_t(x,t) = -\bar{\bar{v}}_{tt}(x,t)-\Delta^2 \bar{\bar{v}}(x,t)-\Delta^2 \bar{\bar{v}}_t(x,t) & \mbox{on } \Gamma_0\times[0,T] \\
\bar{\bar{v}}(x,t)=\frac{\pa \bar{\bar{v}}}{\pa \nu}(x,t)=0 & \mbox{on } \pa\Gamma_0\times[0,T] \\
\frac{\pa \bar{\bar{z}}}{\pa \nu}(x,t)=\bar{\bar{v}}_t(x,t) & \mbox{on } \Gamma_0\times[0,T] \\
\bar{\bar{z}}(\cdot,\frac{T}{2}) = \Delta z_0(x)+q(x)z_0 & \mbox{in } \Omega \\
\bar{\bar{z}}_t(\cdot,\frac{T}{2}) = \Delta z_1(x)+q(x)z_1 & \mbox{in } \Omega \\
\bar{\bar{v}}(\cdot,\frac{T}{2}) = -z_1(x)-\Delta^2 v_0(x)-\Delta^2 v_1(x) & \mbox{on } \Gamma_0 \\
\bar{\bar{v}}_t(\cdot,\frac{T}{2}) = -\Delta z_0(x)-q(x)z_0(x)-\Delta^2 v_1(x) \\
                                     \hspace*{1.0 in} +\Delta^2 z_1(x)+\Delta^4 v_0(x)+\Delta^4 v_1(x) & \mbox{on } \Gamma_0
\end{cases}
\end{equation}
Then still using the similarly operator setting as before we can compute the domain of $\mathcal{A}^3$:
\begin{equation}\label{DomA3}
\begin{split}
D(\mathcal{A}^3)& =\{[z_0,z_1,v_0,v_1]^T:(z_1,(-A_N+q)z_0+Bv_1,v_1,-\textbf{\AA}(v_0+v_1)-B^{*}z_1\in D(\mathcal{A}^2)\} \\
                & =\{[z_0,z_1,v_0,v_1]^T: z_1\in H^3_{\Gamma_1}(\Omega), (\Delta+q)z_0\in H^2_{\Gamma_1}(\Omega), \frac{\pa z_0}{\pa\nu}|_{\Gamma_0}=v_1, \\
& v_0\in H^2_0(\Gamma_0), v_1\in H^2_0(\Gamma_0), \textbf{\AA}(v_0+v_1)+B^{*}z_1\in H^2_0(\Gamma_0), \\
& \textbf{\AA}(v_1-\textbf{\AA}(v_0+v_1)-B^{*}z_1)+B^{*}[(-A_N+q)z_0+Bv_1]\in H^2_0(\Gamma_0), \\
& \textbf{\AA}(v_0+v_1)+B^{*}z_1+\textbf{\AA}[v_1-\textbf{\AA}(v_0+v_1)-B^{*}z_1]+B^{*}[(-A_N+q)z_0+Bv_1]\in D(\textbf{\AA}) \\
& \frac{\pa z_1}{\pa\nu}|_{\Gamma_0}=-\textbf{\AA}(v_0+v_1)-B^{*}z_1, \frac{\pa [(-A_N+q)z_0+Bv_1]}{\pa\nu}|_{\Gamma_0}= \\
& -\textbf{\AA}[v_1-\textbf{\AA}(v_0+v_1)-B^{*}z_1]-B^{*}[(-A_N+q)z_0+Bv_1]\} \\
                & =\{[z_0,z_1,v_0,v_1]^T: z_0\in H^{\frac{7}{2}}_{\Gamma_1}(\Omega), z_1\in H^3_{\Gamma_1}(\Omega), v_0\in H^2_0(\Gamma_0), v_1\in H^2_0(\Gamma_0), \\
& \textbf{\AA}(v_0+v_1)+B^{*}z_1\in H^2_0(\Gamma_0), \frac{\pa z_0}{\pa\nu}|_{\Gamma_0}=v_1, \frac{\pa z_1}{\pa\nu}|_{\Gamma_0}=-\textbf{\AA}(v_0+v_1)-B^{*}z_1 \ \text{on} \ \Gamma_0, \\
& \textbf{\AA}(v_0+v_1)+B^{*}z_1+\textbf{\AA}[v_1-\textbf{\AA}(v_0+v_1)-B^{*}z_1]+B^{*}[(-A_N+q)z_0+Bv_1]\in D(\textbf{\AA}) \\
& \textbf{\AA}(v_1-\textbf{\AA}(v_0+v_1)-B^{*}z_1)+B^{*}[(-A_N+q)z_0+Bv_1]\in H^2_0(\Gamma_0), \\
& \frac{\pa [(-A_N+q)z_0+Bv_1]}{\pa\nu}|_{\Gamma_0}=-\textbf{\AA}[v_1-\textbf{\AA}(v_0+v_1)-B^{*}z_1]-B^{*}[(-A_N+q)z_0+Bv_1]\}
\end{split}
\end{equation}
where in the last step $z_0\in H^{\frac{7}{2}}_{\Gamma_1}(\Omega)$ is from trace theory of solving $\frac{\pa z_0}{\pa\nu}|_{\Gamma_0}=v_1\in H^{2}(\Gamma_0)$.
Therefore when $z_0\in H^{\frac{7}{2}}_{\Gamma_1}(\Omega)$, $z_1\in H^3_{\Gamma_1}(\Omega)$, $v_0\in H^2_0(\Gamma_0)$, $v_1\in H^2_0(\Gamma_0)$
with compatible conditions as in $D(\mathcal{A}^3)$ and $q\in W^{2,\infty}(\Omega)$, then from semigroup theory we have that the solution of $\eqref{nonlinear2}$ satisfies
\be\label{barbarztztt}
\bar{\bar{z}}_t\in C([0,T];H^1(\Omega)), \quad \bar{\bar{z}}_{tt}\in C([0,T];L^2(\Omega))
\ee
Hence we have on the one hand
\be\label{regzt21}
z_t\in H^2(0,T;H^1(\Omega))
\ee
On the other hand, from $\eqref{barbarztztt}$ and $\bar{\bar{z}}_{tt}(x,t) = \Delta \bar{\bar{z}}(x,t) + q(x)\bar{\bar{z}}(x,t)$, we have by elliptic theory that
\be
z_{tt}=\bar{\bar{z}}\in L^2(0,T;H^2(\Omega))
\ee
which implies
\be\label{regzt12}
z_t\in H^1(0,T;H^2(\Omega))
\ee
Now interpolate between $\eqref{regzt21}$ and $\eqref{regzt12}$, we have for $0<\epsilon<\frac{1}{2}$,
\be
z_t\in H^{\frac{3}{2}}(0,T;H^{\frac{3}{2}}(\Omega))\subset H^{\frac{1}{2}+\epsilon}(0,T;L^{\infty}(\Omega))
\ee
where the inclusion is again by Sobolev embedding.

Hence in either case $n=2$ or $n=3$, under the assumptions on the initial data $[z_0,z_1,v_0,v_1]$ and $q(x)$, $p(x)$ in Theorem \ref{th4},
we have that $z_t\in H^{\frac{1}{2}+\epsilon}(0,T;L^{\infty}(\Omega))$. Thus when we again set
$f(x)=q(x)-p(x)$, $w(x,t)=z(q)(x,t)-z(p)(x,t)$, $u(x,t)=v(q)(x,t)-v(p)(x,t)$ and $R(x,t)=z(p)(x,t)$
as in section 4, we obtain $\eqref{regRt}$ that $R_{t}\in H^{\frac{1}{2}+\epsilon}(0,T;L^{\infty}(\Omega))$ and hence
all the assumptions in theorem \ref{th3} are satisfied. Therefore, we get the desired stability $\eqref{nonlinearstability}$ from the stability $\eqref{stability}$ of
the linear inverse problem in Theorem \ref{th3}. \hfill $\Box$

\section{Concluding remark}
As we mentioned at the beginning and the calculations of $D(\mathcal{A}^2)$ and $D(\mathcal{A}^3)$ show, the lack of compactness
of the resolvent limits the space regularity of the solutions for the wave equation parts since we always have the elliptic
problem for $z$ or $z_t$ such that $(\Delta+q)z\in L^2(\Omega)$ with $\frac{\pa z}{\pa\nu}\in H^2_0(\Gamma_0)$ provided $q$ in some suitable space.
Therefore the best space regularity that $z$ could get is $2+\frac{3}{2}=\frac{7}{2}$ from elliptic and trace theory. As a result, our argument of the
stability in the nonlinear inverse problem will only work for dimension up to $n=7$ as we need the Sobolev embedding $H^{\frac{n}{2}}(\Omega)\subset L^{\infty}(\Omega)$
in order to achieve the space regularity of $z_t$ in $L^{\infty}(\Omega)$ which is needed in the proof.

\end{document}